\documentclass[11pt,a4paper]{amsart}

\usepackage{amssymb,xspace}
\usepackage{amstext}
\theoremstyle{plain}
\usepackage{amsbsy,amssymb,amsfonts,latexsym}
\usepackage[utf8]{inputenc}
\usepackage{xcolor}

\usepackage{pgf,tikz}
\usepackage{mathrsfs}
\usetikzlibrary{arrows}

\usepackage{enumerate}
\usepackage{graphics}
\usepackage{subcaption}

\date{\today}
\title{Common hypercyclic vectors and dimension of the parameter set}
\author{Fr\'ed\'eric Bayart}
\author{Fernando Costa Jr.}
\address[F. Bayart, F. Costa]{Laboratoire de Mathématiques Blaise Pascal UMR 6620 CNRS, Université Clermont Auvergne, Campus universitaire des Cézeaux, 3 place Vasarely, 63178 Aubière Cedex, France.} 
\email{Frederic.Bayart@uca.fr, Fernando.Vieira\underline{ }Costa\underline{ }Junior@uca.fr}
\author[Q. Menet]{Quentin Menet}
\address[Q. Menet]{Service de Probabilit\'e et Statistique, D\'epartement de Math\'ematique\\ Universit\'{e} de Mons\\ Place du Parc 20\\ 7000 Mons (Belgium)}
\email{quentin.menet@umons.ac.be}
 \thanks{This work was supported in part by
the project FRONT of the French
National Research Agency (grant ANR-17-CE40-0021). The third author is a Research Associate of the Fonds de la Recherche Scientifique - FNRS}

\newcommand{\veps}{\varepsilon}

\def\RR{\mathbb R}

\def\NN{\mathbb N}

\def\CC{\mathbb C}

\def\card{\textrm{card}}
\def\dboxsup{\overline{\dim_{\rm{B}}}}
\def\dimh{\dim_{\mathcal H}}

\def\bj{\mathbf j}

\DeclareMathOperator{\diam}{diam}

\def\bk{\mathbf k}
\def\bj{\mathbf j}
\def\Imr{I_r^m}
\def\dimh{\dim_{\mathcal H}}
\def\dboxsup{\overline{\dim}_{\rm B}}
\def\dboxhom{\dim_{\rm HB}}

\newtheorem{theorem}{Theorem}[section]

\newtheorem{lemma}[theorem]{Lemma}

\newtheorem{proposition}[theorem]{Proposition}

\newtheorem{corollary}[theorem]{Corollary}

{\theoremstyle{definition}}
{\theoremstyle{definition}}

{\theoremstyle{definition}\newtheorem{example}[theorem]{Example}}

{\theoremstyle{definition}\newtheorem{definition}[theorem]{Definition}}

{\theoremstyle{definition}}

{\theoremstyle{definition}\newtheorem{remark}[theorem]{Remark}}

\newtheorem{question}[theorem]{Question}

\newtheorem*{CSTHEOREM}{Costakis-Sambarino Theorem}

\begin{document}

\definecolor{zzttqq}{rgb}{0.6,0.2,0}
\definecolor{zzttqqa}{rgb}{0.2,0.6,0}
\definecolor{cqcqcq}{rgb}{0.7529411764705882,0.7529411764705882,0.7529411764705882}

\begin{abstract}
We investigate the existence of a common hypercyclic vector for a family $(T_\lambda)_{\lambda\in \Lambda}$ of hypercyclic
operators acting on the same Banach space $X$. We give positive and negative results involving the dimension of $\Lambda$ and the regularity of each map $\lambda\in \Lambda\mapsto T_\lambda^n x$, $x\in X$, $n\in\mathbb N$.
\end{abstract}

\maketitle


\section{Introduction}

Among the many problems arising in linear dynamics, that of finding a common hypercyclic vector for an uncountable
family of hypercyclic operators is one of the most prominent. Let us introduce the relevant definitions. Let $X$ be an infinite-dimensional and separable $F$-space
and let $T\in\mathcal L(X)$. A vector $x\in X$ is said to be {\it hypercyclic} for $T$ if its orbit under $T$, $\{T^n x:\ n\geq 0\}$ is
dense in $X$. The set of hypercyclic vectors for $T$ will be denoted by $HC(T)$. We refer to the two books \cite{BM09} and \cite{GePeBook}
for the standard theory of hypercyclic operators. 

Given $(T_\lambda)_{\lambda\in\Lambda}$ a family of hypercyclic operators acting on the same $F$-space, it is natural to ask
whether $\bigcap_{\lambda\in\Lambda}HC(T_\lambda)$ is nonempty. The first result in that direction is due to Abakumov
and Gordon who showed in \cite{AG} that $\bigcap_{a>0}HC(e^aB)$ is nonempty, where $B$ is the unweighted backward shift
acting on $\ell_p$, $p\in[1,+\infty)$ or on $c_0$. Soon after, Costakis and Sambarino in \cite{CoSa04a} came with a criterion for proving 
the common hypercyclicity of some families, which allow them to extend the results of \cite{AG} to other families of shifts
or to translation operators. 

The paper \cite{AG} also contains an important negative result, granted to Borichev: the two-dimensional family $(e^a B\times e^b B)_{(a,b)\in(0,+\infty)^2}$ acting on $\ell_2\times\ell_2$, does not admit a common hypercyclic vector. 
It turns out that most of the examples of families admitting a common hypercyclic vector are one-dimensional families,
with two notable exceptions: the Leon-M\"uller theorem \cite{LeMu04} which allows to introduce an extra parameter of rotations, and translation operators which have some redundant properties (see for instance \cite{Ba15}). Even for one-dimensional families, several
intriguing problems remain. For instance, if $\Lambda\subset(0,+\infty)^2$ is a monotonic Lipschitz curve, then $(e^{a}B\times e^b B)_{(a,b)\in\Lambda}$ possesses a common hypercyclic vector, whereas this is unknown for $(e^{t}B\times e^{3-t}B)_{t\in[1,2]}$ (see \cite{BMIND}).

Our ambition, in this paper, is to revisit this problem and to shed new light on common hypercyclic vectors. We begin with a review of the results we intend to prove. 
In what follows, the parameter set $\Lambda$ will always be a subset of $\RR^d$ for some $d\geq 1$ and $\RR^d$ will be endowed with the sup-norm.


\subsection{Products of multiples of the backward shift}

Our first result is an answer to the problem of \cite{BMIND} we just recalled. More precisely, we will prove the following.

\begin{theorem}\label{thm:lipschitz}
Let $X=\ell_p(\mathbb N)$, $p\in[1,+\infty)$, or $X=c_0(\mathbb N)$, and let $\Lambda\subset (0,+\infty)^d$ be a Lipschitz curve. Then $(e^{\lambda(1)}B\times\cdots\times e^{\lambda(d)} B)_{\lambda\in\Lambda}$ possesses a dense $G_\delta$ set of common hypercyclic vectors.
\end{theorem}

The way to delete the assumption "$\Lambda$ is monotonic" in Theorem \ref{thm:lipschitz} will be to obtain a characterization for the common hypercyclicity of a family of products of weighted shifts acting on a Fr\'echet sequence space, when these
shifts satisfy some natural conditions. This condition, which is rather technical, takes a much more pleasant form when we apply it to multiples of the backward shifts. We will apply it in order to get Theorem \ref{thm:lipschitz}.


\subsection{Borichev result revisited}

Borichev's result can be rephrased in the following more precise way:
\begin{quote}
Let $\Lambda\subset(0,+\infty)^2$ be such that $\bigcap_{(a,b)\in\Lambda}HC(e^a B\times e^b B)$ is not empty. Then $\Lambda$ has measure zero.
\end{quote}
 If we analyze
the proof of this result, it turns out that a key point is given by the fact that if $(e^a B)^nu$ and $(e^{a'}B)^nu$ are both close to the same nonzero vector (e.g. $e_0$), then $|a-a'|$ has to be small, precisely $|a-a'|\leq C/n$ for some constant $C$.

We will show that this can be put in a more general framework, replacing sets of zero Lebesgue measure by sets of small Hausdorff dimension.

\begin{theorem}\label{thm:hausdorffintro}
Let $\Lambda\subset\RR^d$, let $(T_\lambda)_{\lambda\in\Lambda}$ be a family of operators acting on the Banach space $X$. Assume that there exist $\alpha>0$, $v\in X$, $\delta>0$ and $C>0$ such that,
for all $\lambda,\mu\in\Lambda$, for all $n\in\NN$ and all $u\in X$ satisfying
$$\|T^n_\lambda u-v\|<\delta\textrm{ and }\|T^n_\mu u-v\|<\delta,$$
then 
$$\|T_\lambda^n u -T_\mu^n u\|\geq Cn^{\alpha} \|\lambda-\mu\|.$$
If $\bigcap_{\lambda\in\Lambda}HC(T_\lambda)\neq\varnothing$, then $\dimh(\Lambda)\leq \frac 1\alpha$.
\end{theorem}

In particular, this can be applied to the family $(B_{w(\lambda(1))}\times\cdots \times B_{w(\lambda(d))})_{\lambda\in (0,+\infty)^d}$ where $(w(a))_{a>0}$ is defined by $w_1(a)\cdots w_n(a)=\exp(an^\alpha)$ for some $\alpha>0$ and for all $n\in\NN$.
\begin{corollary}\label{cor:sizenalpha}
Let $X=\ell_p(\mathbb N)$, $p\in[1,+\infty)$ or $X=c_0(\mathbb N)$. Let $\alpha\in(0,1]$, let $(w(a))_{a>0}$ be defined by $w_1(a)\cdots w_n(a)=\exp(an^\alpha)$ for all $n\in\NN$
and let $\Lambda\subset (0,+\infty)^d$. If $(B_{w(\lambda(1))}\times\cdots\times B_{w(\lambda(d))})_{\lambda\in\Lambda}$ admits a common hypercyclic vector, then $\dimh(\Lambda)\leq 1/\alpha$.
\end{corollary}

In view of the previous corollary, one may ask if the converse holds true, namely if the condition $\dimh(\Lambda)\leq 1/\alpha$, or $\dimh(\Lambda)<1/\alpha$, implies that 
$\bigcap_{\lambda\in\Lambda}HC(B_{w(\lambda(1))}\times\cdots\times B_{w(\lambda(d))})\neq\varnothing$. 
More specifically, we may ask if $\alpha\leq 1/d$ implies $\bigcap_{\lambda\in(0,+\infty)^d}HC(B_{w(\lambda(1))}\times\cdots\times B_{w(\lambda(d))})\neq\varnothing$.
The study of these questions motivates the remaining part of the paper.


\subsection{A common hypercyclicity criterion in dimension greater than $1$}

In the remaining of this introduction, we will always assume that $X$ is a separable Banach space.
We discuss now common hypercyclicity criteria for a family $(T_\lambda)_{\lambda\in\Lambda}$ of operators acting on the same Banach space $X$. 
 We will always assume that the following assumptions are true:
\begin{itemize}
\item the map $(\lambda,u)\mapsto T_\lambda u$ is continuous from $\Lambda\times X$ into $X$;
\item there exists a dense set $\mathcal D\subset X$ such that each operator $T_\lambda$ has a partial right-inverse $S_\lambda:\mathcal D\to\mathcal D$, that is $T_\lambda S_\lambda(u)=u$ for all $u\in\mathcal D$ and all $\lambda\in\Lambda$.
\end{itemize}

These assumptions are for instance satisfied if $T_\lambda$ is defined as the product of weighted shifts $B_{w^{(1)}(\lambda(1))}\times\cdots\times B_{w^{(d)}(\lambda(d))}$, $X=c_0^d(\NN)$ or $X=\ell_p^d(\NN)$, $p\in[1,+\infty)$, and, for each $i=1,\dots,d$ and each $n\in\NN$, $a\mapsto w_n^{(i)}(a)$ is continuous (we will call this a \emph{continuous family of weights}).

When $\Lambda$ is an interval of the real line, one of the most useful result to get common hypercyclic vectors is the Costakis-Sambarino criterion:

\begin{CSTHEOREM}
Let $\Lambda$ be an interval of the real line. Assume that for every compact interval $K\subset \Lambda$, every $u\in \mathcal{D}$,
\begin{enumerate}
\item[(CS1)] there exist $m\geq 1$ and $(c_k)_{k\ge m}$ a summable sequence of positive real numbers such that
\begin{enumerate}
\item $\|T_{\lambda}^{n+k}S_{\mu}^n u||\le c_k$ for every $n\in\NN$, $k\ge m$, $\mu\le \lambda$, $\mu,\lambda\in K$
\item $\|T_{\lambda}^{n}S_{\mu}^{n+k} u||\le c_k$ for every $n\in\NN$, $k\ge m$, $\mu\ge \lambda$,  $\mu,\lambda\in K$;
\end{enumerate}
\item[(CS2)] for all $\veps>0$, there exists $\tau>0$ such that, for all $n\geq 1$, 
\[0\le \mu-\lambda\le \frac{\tau}{n} \Rightarrow \|T_{\lambda}^nS_{\mu}^n(u)-u\|<\veps.\]
\end{enumerate}
Then $\bigcap_{\lambda\in \Lambda} HC(T_{\lambda})$ is a dense $G_{\delta}$ subset of $X$.
\end{CSTHEOREM}

We look for a substitute for this theorem when $\Lambda$ is not an interval of the real line and in particular if the "dimension" of $\Lambda$ is greater than $1$. The continuity condition (CS2) is naturally implied by the following Lipschitz estimate: for all $u\in\mathcal D$, there exists $C>0$ such that, for all $n\geq 1$, for all $\lambda,\mu\in\Lambda$, 
$$\|T_\lambda^n S_\mu^n u-T_\mu^nS_\mu^n u\|\leq Cn\|\lambda-\mu\|.$$
Nevertheless, if we have the opposite inequality
$$\|T_\lambda^n S_\mu^n u-T_\mu^nS_\mu^n u\|\geq Cn\|\lambda-\mu\|,$$
then Theorem \ref{thm:hausdorffintro} essentially says that the Hausdorff dimension of the set of common hypercyclic vectors cannot exceed 1.
Hence, to get common hypercyclic vectors for subsets of $\RR^d$ of bigger dimension, we will need a stronger condition, at least something like
\[\| \lambda-\mu\| \le \frac{\tau}{n^\alpha} \Rightarrow \|T_{\lambda}^nS_{\mu}^n(u)-u\|<\veps\]
for some $\alpha\in(0,1)$. Under this last condition and an appropriate substitute for (CS1), we will be able to prove a common hypercyclic criterion when 
$\Lambda\subset\RR^d$ possesses some regularity and has "dimension" less than $1/\alpha$. The notion of dimension we need 
is a kind of homogeneous box dimension. For $r\geq 1$, we define $I_r=\{1,\dots,r\}$.

\begin{definition}\label{def:homogeneousdimension}
Let $\Lambda\subset\RR^d$ be compact. We say that $\Lambda$ has \emph{homogeneous box dimension at most $\gamma\in (0,d]$}
if there exist $r\geq 2$, $C(\Lambda)>0$ and, for all $m\geq 1$, a family $(\Lambda_\bk)_{\bk\in I_r^m}$ of compact subsets of $\Lambda$  such that for all $m\geq 1$,
\begin{itemize}
\item for all $\bk\in I_r^m$, $\diam(\Lambda_\bk)\leq C(\Lambda)\left(\frac1{r^{1/\gamma}}\right)^m$;
\item $\Lambda\subset \bigcup_{\bk\in I_r^m}\Lambda_\bk$;
\item for all $\bk\in I_r^m$, $\Lambda_{k_1,\dots,k_m}\subset \Lambda_{k_1,\dots,k_{m-1}}$.
\end{itemize}
The \emph{homogeneous box dimension} of $\Lambda$ is defined as the infimum of the $\gamma\in(0,d]$ such that $\Lambda$ has homogeneous box dimension at most $\gamma$ and will be denoted $\dboxhom(\Lambda)$.
\end{definition} 

We will discuss later the link between this notion of dimension and more classical ones; we just observe for the moment
that any compact subset of $\RR^d$ has homogeneous box dimension at most $d$. 
Having this notion of dimension at hand, we can prove the following result. 

\begin{theorem}\label{thm:multiprecised}
Let $\gamma>0$  and let $\Lambda$ be a compact subset of $\RR^d$ having homogeneous box dimension at most $\gamma$.
Assume moreover that there exist $\alpha\in(0,1/\gamma)$, $\beta>\alpha\gamma$ and $D>0$ such that, for all $u\in\mathcal D$, 
\begin{enumerate}[(a)]
\item there exist $C>0,N>0$ such that, for all $\lambda,\mu\in \Lambda$, for all $n\geq 0$ and $k\geq N$ such that $\|\lambda-\mu\|\leq D\frac{k^\alpha}{(n+k)^\alpha}$, then 
\[ \left\| T_\lambda^{n+k}S_\mu^n u\right\|\leq \frac C {k^\beta}, \]
\[ \left\| T_\lambda^{n}S_\mu^{n+k} u\right\|\leq \frac C {k^\beta}. \]
\item for all $\veps>0$, there exists $\tau>0$ such that, for all $n\geq 1$, for all $\lambda,\mu\in\Lambda$,
\[ \|\lambda-\mu\|\leq \frac{\tau}{n^\alpha}\implies \left\|T_\lambda^n S_\mu^n u-u\right\|<\veps. \]
\end{enumerate}
Then $\bigcap_{\lambda\in\Lambda} HC(T_\lambda)$ is a dense $G_\delta$ subset of $X$.
\end{theorem}

In particular we get the following corollary, which can be seen as the desired converse of Corollary \ref{cor:sizenalpha}.
\begin{corollary}\label{cor:positifnalpha}
Let $X=\ell_p(\NN)$, $p\in[1,+\infty)$ or $X=c_0(\NN)$, $\alpha\in(0,1]$, let $(w(a))_{a>0}$ be defined by $w_1(a)\cdots w_n(a)=\exp(an^\alpha)$ for all $n\in\NN$
and let $\Lambda\subset (0,+\infty)^d$. Assume that $\dboxhom(\Lambda)<1/\alpha$. Then $(B_{w(\lambda(1))}\times\cdots\times B_{w(\lambda(d))})_{\lambda\in\Lambda}$ admits a common hypercyclic vector.
\end{corollary}

\begin{example}\label{ex:expo}
Let  $X=\ell_p(\NN)$, $p\in[1,+\infty)$ or $X=c_0(\NN)$, $\alpha\in (0,1/d)$ and let $(w(a))_{a>0}$ be the family of weights defined by $w_1(a)\cdots w_n(a)=\exp(a n^\alpha)$ (resp. $w_n(a) = 1 + \frac{a}{n^{1-\alpha}}$) for all $n\geq 1$. Then $(B_{w(\lambda(1))}\times\cdots\times B_{w(\lambda(d))})_{\lambda\in (0,+\infty)^d}$ admits a common hypercyclic vector.
\end{example}

Theorem \ref{thm:multiprecised} can also be applied to H\"older curves leading to a nice complement to Theorem~\ref{thm:lipschitz}. Indeed, let $\alpha\in (0,1]$. A compact set $\Lambda\subset\RR^d$ is called an {\it $\alpha$-H\"older curve} if $\Lambda=f(I)$ for some function $f:[0,1]\to \RR^d$ satisfying
$$\exists C>0,\ \forall (s,t)\in[0,1]^2,\ \|f(s)-f(t)\|\leq C|s-t|^\alpha.$$
It turns out that any $\alpha$-H\"older curve has homogeneous box dimension at most $1/\alpha$. Therefore, we will also obtain the following example.

\begin{example}\label{ex:holder}
Let $X=\ell_p(\NN)$, $p\in[1,+\infty)$ or $X=c_0(\NN)$. Let $\Lambda\subset (0,+\infty)^d$ be a $\beta$-H\"older curve for some $\beta\in(0,1]$. Let $\alpha\in (0,\beta)$ and let  $(w(a))_{a>0}$ be the family of weights defined by $w_1(a)\cdots w_n(a)=\exp(a n^\alpha)$ (resp. $w_n(a) = 1 + \frac{a}{n^{1-\alpha}}$) for all $n\geq 1$. Then $(B_{w(\lambda(1))}\times\cdots\times B_{w(\lambda(d))})_{\lambda\in\Lambda}$ admits a common hypercyclic vector.
\end{example} 


\subsection{The Basic Criterion}

Almost all results of common hypercyclicity rely on the construction of a suitable covering of the parameter space $\Lambda$ and on an associated sequence $(n_k)$. What we need is contained in the following basic criterion (see \cite[Lemma 7.12]{BM09}).

\begin{theorem}[Basic Criterion]\label{basic}
Let $\Lambda'$ be a topological space. 
Let $(T_{\lambda})_{\lambda\in \Lambda'}$ be a family of operators on $X$ such that 
\begin{itemize}
\item the map $(\lambda,x)\mapsto T_{\lambda}(x)$ is continuous from $\Lambda'\times X \to X$;
\item there exist a dense set $\mathcal{D}\subset X$ and maps $S_{\lambda}:\mathcal{D}\to \mathcal{D}$ satisfying $T_{\lambda}S_{\lambda}x=x$ for every $x\in \mathcal{D}$ and every $\lambda\in\Lambda'$.
\end{itemize}
If $\Lambda\subset \Lambda'$ is $\sigma$-compact and if for every compact set $K\subset \Lambda$, every pair $(u,v)\in \mathcal{D}\times\mathcal{D}$, every $\varepsilon>0$, there exist $\lambda_1,\dots,\lambda_q\in \Lambda'$, sets $\Lambda_1,\dots,\Lambda_q\subset \Lambda$ and positive integers $n_1,\dots,n_q$ such that
\begin{enumerate}
\item[(BC1)] $\bigcup_k \Lambda_k\supset K$,
\item[(BC2)] $ \|\sum_k S_{\lambda_k}^{n_k}v\|<\varepsilon$,
\item[(BC3)] for all $k=1,\dots,q$ and all $\lambda\in\Lambda_k$, $\|\sum_{j\ne k} T_{\lambda}^{n_k}S_{\lambda_j}^{n_j}v\|<\varepsilon$,
\item[(BC4)] for all $k=1,\dots,q$ and all $\lambda \in \Lambda_k$, $\|T_{\lambda}^{n_k}u\|<\varepsilon$,
\item[(BC5)] for all $k=1,\dots,q$ and all $\lambda\in\Lambda_k$,  $\|T_{\lambda}^{n_k}S_{\lambda_k}^{n_k}v-v\|<\varepsilon$,
\end{enumerate}
then $\bigcap_{\lambda\in \Lambda} HC(T_{\lambda})$ is a dense $G_{\delta}$ subset of $X$.
\end{theorem}


\section{Characterization of families of products of weighted shifts admitting a common hypercyclic vector}
\label{sec:ws}

In this section, we work in the context of a Fréchet sequence space $X$, namely $X$ is a Fréchet space endowed with a family of seminorms $(\|\cdot\|_p)$, contained in 
the space $\mathbb C^\mathbb N$ of all complex sequences and such that each coordinate functional $(x_n)_n\mapsto x_m$ is continuous. Such a space can be endowed with an $F$-norm $\|\cdot\|$ defining the topology of $X$ (see \cite[Section 2.1]{GePeBook}).
Such an $F$-norm can be defined by the formula
\[ \|x\|=\sum_{p=1}^{+\infty}\frac1{2^p}\min(1,\|x\|_p). \]
In particular, an $F$-norm satisfies the triangle inequality and the inequality
\begin{equation}
\forall \lambda\in\CC,\ \forall x\in X,\ \|\lambda x\|\leq (|\lambda|+1)\|x\|,
\end{equation}
a property which replaces the positive homogeneity of the norm.
We will also need the following property of a Fréchet sequence space in which $(e_n)$ is an unconditional basis (see \cite[Theorem 3.9]{GuKa81}).

\begin{enumerate}
\item[(UNC)] If $(x_n)\in X$ and $(\alpha_n)\in\ell_\infty$, then $(\alpha_n x_n)\in X$. Moreover, for all $\veps>0$, for all  $M>0$, there exists $\delta>0$ such that, 
for all $(x_n)\in X$ with $\|(x_n)\|\leq \delta$, for all sequence $(\alpha_n)\in\ell_\infty$ with $\|(\alpha_n)\|_\infty\leq M$, then $(\alpha_n x_n)\in X$ and $\|(\alpha_n x_n)\|<\veps$.
\end{enumerate}

\begin{theorem}\label{thm:carac}
Let $X$ be a Fréchet sequence space in which $(e_n)$ is an unconditional basis. Let $I\subset\mathbb R$ be a nonempty interval and let $\Lambda\subset I^d$ be $\sigma$-compact. Let $(B_{w(a)})_{a\in I}$ be a continuous family of weighted shifts on $X$ and assume that $a\in I\mapsto w_n(a)$ is nondecreasing. Assume also that there exist a nondecreasing map $F:\mathbb N\to\mathbb N$ and $c,C>0$ such that, for all $n\geq 1$, denoting by $f_n(a)=\sum_{k=1}^n \log(w_k(a))$, 
\[ \forall (a,b)\in I^2,\ cF(n) |a-b|\leq |f_n(a)-f_n(b)|\leq CF(n) |a-b| \]
\[ \forall (a,b)\in I^2,\ \frac{w_n(a)}{w_n(b)}\geq c. \]
Then the following assertions are equivalent:
\begin{enumerate}[(a)]
\item $(B_{w(\lambda(1))}\times\cdots \times B_{w(\lambda(d))})_{\lambda\in\Lambda}$ possesses a dense $G_{\delta}$ set of common hypercyclic vectors in $X^d$.
\item $(B_{w(\lambda(1))}\times\cdots \times B_{w(\lambda(d))})_{\lambda\in\Lambda}$ admits a common hypercyclic vector in $X^d$.
\item For all $\tau>0$, for all $N\geq 1$, for all $\veps>0$, for all $K\subset\Lambda$ compact, there exist
$N\leq n_1<n_1+N\leq n_2<\dots<n_{q-1}+N\leq n_{q}$ and $(\lambda_k)_{k=1,\dots,q}\in I^d$ such that
\begin{enumerate}[(i)]
\item $K\subset \bigcup_{k=1}^q \prod_{i=1}^d \left[\lambda_k(i)-\frac\tau{F(n_k)},\lambda_k(i)\right]$
\item For all $i=1,\dots,d$,
\[ \left\|\sum_{k=1}^q\frac{1}{w_1(\lambda_k(i))\cdots w_{n_k}(\lambda_k(i))}e_{n_k}\right\|<\veps. \]
\item For all $k=1,\dots,q-1$, for all $i=1,\dots,d$, for all $l=0,\dots,N$,
\[ \left\|\sum_{j=k+1}^{q}\frac{w_{n_j-n_k+l+1}(\lambda_k(i))\cdots w_{n_j+l}(\lambda_k(i))}{w_{l+1}(\lambda_j(i))\cdots
w_{n_j+l}(\lambda_j(i))}e_{n_j-n_k+l}\right\|<\veps. \]
\end{enumerate}
\end{enumerate}
\end{theorem}

To simplify the notations, we will do the proof only for $d=2$ and we shall denote by $\lambda=(a,b)$ any element of $\mathbb R^2$. For $\lambda=(a,b)\in I^2$, we shall denote by $T_\lambda$ the operator $B_{w(a)}\times B_{w(b)}$ acting on $X\times X$ and by $S_\lambda$ the operator $F_{w^{-1}(a)}\times F_{w^{-1}(b)}$, where $F_{w^{-1}(a)}$ denotes the forward shift associated to the sequence $(w_n^{-1}(a))_n$.

\begin{proof}
We first assume that $(T_\lambda)_{\lambda\in\Lambda}$ admits a common hypercyclic vector. Let $\tau>0$, $N\geq 1$, $\veps>0$ and $K\subset\Lambda$ compact. We set $K_1$ the projection of $K$ onto the first coordinate; $K_1$ is a compact subset of $I$. We consider $0<\eta<\min\left(\frac 12,\frac{c\tau}4\right)$ satisfying also the following two technical conditions:
\begin{equation}\label{eq:caracF1}
\eta<\frac 14\times \frac1{\prod_{k=1}^N \max\left(1,\sup_{a\in K_1}w_k(a)\right)}
\end{equation}
\begin{equation}\label{eq:caracF2}
c\inf_{\substack{a,a'\in K_1\\ l,l'\in\{1,\dots,N+1\}}}\frac{w_l(a)}{w_{l'}(a')}\times\frac{1-\eta}{1+\eta}>\frac{\eta}{1-\eta}.
\end{equation}
By continuity of the first $N+2$ coordinate functionals and by (UNC), we may find $\delta>0$ such that, for all $z=(z_n)\in X$ and all $\alpha=(\alpha_n)\in\ell_\infty$, 
\[ \|(z_n)\|\leq\delta\implies \forall l\in\{0,\dots,N+1\},\ |z_l|<\eta \]
\[ \|(z_n)\|\leq\delta\textrm{ and }\|(\alpha_n)\|\leq 2\implies \|(\alpha_n z_n)\|<\veps. \]

Let $u=(x,y)$ be a common hypercyclic vector for $(T_\lambda)_{\lambda\in\Lambda}$. We may always assume that $\max\{\|x\|,\|y\|\}\leq \delta$ and we set 
\[ v=\left(\sum_{j=0}^N e_j,\sum_{j=0}^N e_j\right). \]
Let $(n_k)$ be an increasing enumeration of 
\[\left\{n\geq 1:\ \left \|T_\lambda^n u-v\right\|<\delta\textrm{ for some }\lambda\in K\right\}. \]
Let $\Lambda_k=\left\{\lambda\in K:\ \left\|T_\lambda^{n_k}u-v\right\|<\delta\right\}$. Since $u$ is a common hypercyclic vector for the family $(T_\lambda)_{\lambda\in K}$, since $K$ is compact and each $\Lambda_k$ is open, there exists $q\geq 1$ such that 
$$K\subset \bigcup_{k=1}^q \Lambda_k.$$
For each $k=1,\dots,q$, we define $a_k$  and $b_k$ by 
\begin{align*}
a_k&:=\sup\left\{a:\ \exists b,\ (a,b)\in \Lambda_k\right\}\\
b_k&:=\sup\left\{b:\ \exists a,\ (a,b)\in \Lambda_k\right\}
\end{align*}
and we set $\lambda_k=(a_k,b_k)$. We first observe that $n_1\geq N$. Indeed, since $\|B_{w(a_1)}^{n_1}x-\sum_{j=0}^N e_j\|\leq \delta$, we know that 
$$w_1(a_1)\cdots w_{n_1}(a_1)|x_{n_1}|\geq 1-\eta>1/2.$$
Assume that $n_1<N$. Then 
$$w_1(a_1)\cdots w_{n_1}(a_1)\leq \prod_{k=1}^N \max\left(1,\sup_{a\in K_1}w_k(a)\right)$$
whereas 
\begin{align*}
w_1(a_1)\cdots w_{n_1}(a_1)&\geq w_1(a_1)\cdots w_{n_1}(a_1)\frac{|x_{n_1}|}{\eta}\\
&>\frac 1{2\eta},
\end{align*}
which contradicts \eqref{eq:caracF1}. We now show that we also have $n_{k+1}-n_k\ge N$ for all $k=1,\dots,q-1$. On the contrary, assume that there exists some $k$ such that $n_{k+1}-n_k< N$. We set $p=n_k+N-n_{k+1}\in\{1,\dots,N-1\}$. 
Then using that $\|B^{n_{k+1}}_{w(a_{k+1})}x-\sum_{j=0}^Ne_j\|\le \delta$, we get
$$\left\{
\begin{array}{rcl}
|w_{p+1}(a_{k+1})\cdots w_{n_k+N}(a_{k+1})x_{n_k+N}-1|&<&\eta\\
|w_{p+2}(a_{k+1})\cdots w_{n_k+N+1}(a_{k+1})x_{n_k+N+1}-1|&<&\eta.
\end{array}\right.$$
Since we also know that $\|B^{n_{k}}_{w(a_{k})}x-\sum_{j=0}^Ne_j\|\le \delta$, we also get
$$\left\{
\begin{array}{rcl}
|w_{N+1}(a_{k})\cdots w_{n_k+N}(a_{k})x_{n_k+N}-1|&<&\eta\\
|w_{N+2}(a_{k})\cdots w_{n_k+N+1}(a_{k})x_{n_k+N+1}|&<&\eta.
\end{array}\right.$$
Taking respective quotients, these inequalities lead to 
$$\frac{w_{N+2}(a_{k})\cdots w_{n_k+N+1}(a_{k})}{w_{p+2}(a_{k+1})\cdots w_{n_k+N+1}(a_{k+1})}\leq\frac{\eta}{1-\eta}$$
$$\frac{w_{N+1}(a_{k})\cdots w_{n_k+N}(a_{k})}{w_{p+1}(a_{k+1})\cdots w_{n_k+N}(a_{k+1})}\geq \frac{1-\eta}{1+\eta}.$$
Consequently we obtain
$$\frac{1-\eta}{1+\eta}\times\frac{w_{n_k+N+1}(a_k)}{w_{n_k+N+1}(a_{k+1})}\times \frac{w_{p+1}(a_{k+1})}{w_{N+1}(a_k)}\leq\frac{\eta}{1-\eta}.$$
This again leads to a contradiction, with \eqref{eq:caracF2}.

Let us now prove (i). We consider $\lambda=(a,b)\in\Lambda_k$ for some $k=1,\dots,q$. The choice of $\delta$ ensures that for any $0\le l\le N$,
$$\left\{
\begin{array}{rcl}
|w_{l+1}(a_k)\cdots w_{n_k+l}(a_k)x_{n_k+l} -1|&<&\eta\\
|w_{l+1}(a)\cdots w_{n_k+l}(a)x_{n_k+l} -1|&<&\eta.
\end{array}\right.$$
Hence, 
\[ |w_1(a_k)\cdots w_{n_k}(a_k)-w_1(a)\cdots w_{n_k}(a)|\cdot |x_{n_k}|<2\eta. \]
On the other hand,
\begin{align*}
&w_1(a_k)\cdots w_{n_k}(a_k)-w_1(a)\cdots w_{n_k}(a)\\
&\quad\quad =w_1(a)\cdots w_{n_k}(a)\left(
\exp\left(\sum_{j=1}^{n_k}\log(w_j(a_k))-\sum_{j=1}^{n_k}\log(w_j(a))\right)-1\right)\\
&\quad\quad\geq w_1(a)\cdots w_{n_k}(a)\left(\exp\left(cF(n_k) (a_k-a)\right)-1\right)\\
&\quad\quad\geq cw_1(a)\cdots w_{n_k}(a) (a_k-a)F(n_k).
\end{align*}
Since we also know that $w_1(a)\cdots w_{n_k}(a) |x_{n_k}|\geq 1/2$, we finally get 
\[ 0\leq a_k-a<\frac{4\eta}{cF(n_k)}<\frac{\tau}{F(n_k)}. \]
The same is true for the second coordinate and we get (i). In order to prove (ii), we define $(\alpha_n)$ by 
$$\alpha_n=\left\{\begin{array}{cl}
\frac{1}{w_1(a_k)\cdots w_{n_k}(a_k)x_{n_k}}&\textrm{ if $n=n_k$ for some $k=1,\dots,q$}\\
0&\textrm{ otherwise.}\\
\end{array}\right.$$
Observe that $\|\alpha\|_\infty\leq 2$, hence
\[ \left\|\sum_{k=1}^ q \frac1{w_1(a_k)\cdots w_{n_k}(a_k)}e_{n_k}\right\|=\left\|\sum_{n=0}^{+\infty}\alpha_n x_n e_n\right\|<\veps. \]

It remains to prove (iii). We fix $k=1,\dots,q-1$ and $l=0,\dots,N$, and we now set
$$\beta_n=\left\{\begin{array}{cl}
\frac{1}{w_{l+1}(a_j)\cdots w_{n_j+l}(a_j)x_{n_j+l}}&\textrm{ if $n=n_j-n_k+l$ for some $j\geq k+1$},\\
0&\textrm{otherwise}.
\end{array}\right.$$
Again $\|\beta\|_\infty\leq 2$ and writing $B_{w(a_k)}^{n_k}(x)-\sum_{j=0}^N e_j$ as $(z_n)_n$, one gets 
$$\sum_{n}\beta_n z_n e_n=\sum_{j=k+1}^q \frac{w_{n_j-n_k+l+1}(a_k)\cdots w_{n_j+l}(a_k)}{w_{l+1}(a_j)\cdots w_{n_j+l}(a_j)}e_{n_j-n_k+l}$$
since $n_j-n_k>N$ for all $j\geq k+1$. The result follows again from the choice of $\delta$.

\smallskip

It remains to show that $(c)$ implies $(a)$ since $(a)\Rightarrow (b)$ is obvious. To this end, we shall apply the Basic Criterion. Let $K\subset \Lambda$ be compact, let $\mathcal D\subset X^2$ be the set of couples of vectors with finite support and let
$(u,v)\in\mathcal D\times \mathcal{D}$. We shall write
$$u(1)=\sum_{l=0}^N u_l e_l\textrm{ and }v(1)=\sum_{l=0}^N v_l e_l$$
for some $N\geq 0$. We fix $\veps>0$ and $\tau>0$ (conditions on $\veps$ and $\tau$ will be imposed later) and we consider the two sequences $(n_k)_{k=1,\dots,q}$ and $(\lambda_k)_{k=1,\dots,q}$ given by (c) with $\lambda_k:=(a_k,b_k)$. We set $\Lambda_k=K\cap ([a_k-\tau/F(n_k),a_k]\times [b_k-\tau/F(n_k),b_k])$ so that $\bigcup_k \Lambda_k\supset K$ and we show that the assumptions of the Basic Criterion are satisfied for the sequence $m_k=n_k-N$. First, we observe that
$$\left\|\sum_{k=1}^q F_{w^{-1}(a_k)}^{m_k}(v(1))\right\|\leq \sum_{l=0}^N (|v_l|+1)\left\|\sum_{k=1}^q \frac{1}{w_{l+1}(a_k)\cdots w_{n_k-(N-l)}(a_k)}e_{n_k-(N-l)}\right\|.$$
We fix $\tilde a\in I$ and we claim that for all $l=0,\dots,N$,
$$\sum_{k=1}^q \frac{1}{w_{l+1}(a_k)\cdots w_{n_k-(N-l)}(a_k)}e_{n_k-(N-l)}=
B_{w(\tilde a)}^{N-l}\left(\sum_{k=1}^q \frac{x_{k,l}}{w_{1}(a_k)\cdots w_{n_k}(a_k)}e_{n_k}\right)$$
for some sequence $(x_{k,l})_k\in\ell_\infty$ with 
$$\|(x_{k,l})_k\|_\infty\leq \left(\frac{M}c\right)^N,$$
where $M=\max\{1,|w_j(a)|:0\leq j\leq N, (a,b)\in K\textrm{ for some }b\}$. Provided this has been shown, it is easy to adjust $\veps$ so that (BC2) is satisfied, using the continuity of $B_{w(\tilde a)}$ and the unconditionality of $(e_n)$. The proof of the claim follows
from a rather straightforward computation:
\begin{align*}
&\sum_{k=1}^q \frac1{w_{l+1}(a_k)\cdots w_{n_k-(N-l)}(a_k)}e_{n_k-(N-l)}\\
&=\sum_{k=1}^q w_1\cdots w_l(a_k) \frac{w_{n_k-(N-l)+1}(a_k)\cdots w_{n_k}(a_k)}{w_{n_k-(N-l)+1}(\tilde a)\cdots w_{n_k}(\tilde a)}B_{w(\tilde a)}^{N-l}\left(\frac{1}{w_{1}(a_k)\cdots w_{n_k}(a_k)}e_{n_k}\right).
\end{align*}
The proof that the other conditions of the Basic Criterion are satisfied is rather similar. Indeed, for $k=1,\dots,q$ and $(a,b)\in\Lambda_k$, 
\begin{align*}
\sum_{j\neq k} B_{w(a)}^{m_k}F_{w^{-1}(a_j)}^{m_j}(v(1))&=\sum_{j=k+1}^q B_{w(a)}^{m_k}F_{w^{-1}(a_j)}^{m_j}(v(1))\\
&=\sum_{j=k+1}^q \sum_{l=0}^N v_l\frac{w_{n_j-n_k+l+1}(a)\cdots w_{n_j+l-N}(a)}
{w_{l+1}(a_j)\cdots w_{n_j+l-N}(a_j)}e_{n_j-n_k+l}.
\\
\end{align*}
For each $0\le l\le N$, we now set
$$\alpha^{(l)}_n=\left\{\begin{array}{cl}
\frac{w_{n_j-n_k+l+1}(a)\cdots w_{n_j+l}(a)}{w_{n_j-n_k+l+1}(a_k)\cdots w_{n_j+l}(a_k)}\frac{w_{n_j+l-N+1}(a_j)\cdots w_{n_j+l}(a_j)}{w_{n_j+l-N+1}(a)\cdots w_{n_j+l}(a)}&\substack{\textrm{ if $n=n_j-n_k+l$} \\ \text{for some }j\ge k+1},\\
0&\textrm{ otherwise.}\\
\end{array}\right.$$
Therefore, since $\|\alpha^{(l)}\|_{\infty}\le \frac{1}{c^N}$ (because $a\le a_k$) and since 
\begin{align*}
\left\|\sum_{j\neq k} B_{w(a)}^{m_k}F_{w^{-1}(a_j)}^{m_j}(v(1))\right\|
\le \sum_{l=0}^N(|v_l|+1)\left\|\sum_{n=0}^{\infty} \alpha^{(l)}_n\frac{w_{n_j-n_k+l+1}(a_k)\cdots w_{n_j+l}(a_k)}
{w_{l+1}(a_j)\cdots w_{n_j+l}(a_j)}e_{n}\right\|,
\end{align*}
(BC3) follows from the unconditionality of $(e_n)$ if $\varepsilon$ is sufficiently small. 


 We observe that (BC4) is empty and finish the proof by showing (BC5). Let $k=1,\dots,q$ and $(a,b)\in\Lambda_k$. Then 
$$\left\|B_{w(a)}^{m_k}F_{w^{-1}(a_k)}^{m_k}(v(1))-v(1)\right\|
\leq\sum_{l=0}^N (|v_l|+1)\left\|
\left(\frac{w_{l+1}(a)\cdots w_{n_k-(N-l)}(a)}{w_{l+1}(a_k)\cdots w_{n_k-(N-l)}(a_k)}-1\right)e_l\right\|$$
and it is easy to show that this becomes small provided $\tau>0$ becomes small enough,
using that $|a-a_k|\le \frac{\tau}{F(n_k)}$, that $F$ is nondecreasing, that 
\[\frac{w_{l+1}(a)\cdots w_{n_k-(N-l)}(a)}{w_{l+1}(a_k)\cdots w_{n_k-(N-l)}(a_k)}=\frac{\exp(f_{n_k-(N-l)}(a)-f_{n_k-(N-l)}(a_k))}{\exp(f_{l}(a)-f_{l}(a_k))}\]
and that for every $a,b\in I$
$$|f_n(a)-f_n(b)|\leq CF(n) |a-b|.$$
\end{proof}

The previous statement shows clearly that if $(F(n))_n$ grows slowly, there is more hope to get a large set $\Lambda\subset\RR^d$
such that $(T_\lambda)_{\lambda\in\Lambda}$ possesses a common hypercyclic vector. 
Of course, the simplest examples of weights satisfying the assumptions of Theorem \ref{thm:carac} (when $I$ is a bounded interval) are given by $w_1(a)\cdots w_n(a)=\exp(aF(n))$ where $F(n)$ is nondecreasing and $F(n+1)-F(n)$ is bounded, which includes the case of the multiples of the backward shift. For this last example, Theorem \ref{thm:carac} takes the following easier form.

\begin{theorem}\label{thm:caracstandard}
Let $\Lambda\subset(0,+\infty)^d$ be $\sigma$-compact, $X=\ell_p(\mathbb N)$, $p\in[1,+\infty)$ or $X=c_0(\mathbb N)$.
The following assertions are equivalent:
\begin{enumerate}[(a)]
\item $(e^{\lambda(1)}B\times\cdots \times e^{\lambda(d)}B)_{\lambda\in\Lambda}$ admits a common hypercyclic vector.
\item For all $\tau>0$, for all $N\geq 1$, for all $K\subset\Lambda$ compact, there exist
$N\leq n_1<n_1+N\leq n_2<\dots<n_{q-1}+N\leq n_{q}$ and $(\lambda_k)_{k=1,\dots,q}\in (0,+\infty)^d$ such that
\begin{enumerate}[(i)]
\item $K\subset \cup_{k=1}^q \prod_{i=1}^d \left[\lambda_k(i)-\frac\tau{n_k},\lambda_k(i)\right]$
\item for all $k=1,\dots,q-1$, for all $i=1,\dots,d$,
\[ \lambda_{k+1}(i)n_{k+1}-\lambda_k(i)n_k\geq N. \]
\end{enumerate}
\end{enumerate}
\end{theorem}
\begin{proof}
Without loss of generality, we can assume that $\Lambda\subset I^d$ for some bounded interval $I\subset (0,\infty)$. That $(a)\implies (b)$ then follows directly from Theorem \ref{thm:carac}. Indeed, let us apply it for $K\subset\Lambda$ compact, $N\geq 1$ and $\veps=e^{-N}$ to get $(n_k)_{k=1,\dots,q}$ and $(\lambda_k)_{k=1,\dots,q}$. We need only to verify (ii). From (c) (iii) of Theorem \ref{thm:carac} with $l=0$, we know that, for $k=1,\dots,q-1$, for $i=1,\dots,d$,
\[\frac{\exp(n_k \lambda_k(i))}{\exp(n_{k+1}\lambda_{k+1}(i))}=\frac{w_{n_{k+1}-n_k+1}(\lambda_k(i))\cdots w_{n_{k+1}}(\lambda_k(i))}{w_1(\lambda_{k+1}(i))\cdots
w_{n_{k+1}}(\lambda_{k+1}(i))}\leq\veps=e^{-N}\]
and we conclude by taking the logarithm. The converse direction is slightly more difficult. We fix $K\subset\Lambda$ compact, $\tau>0$, $N\geq 1$ and $\veps>0$ and we apply (b) for $K$, $\tau$ and $N_0\geq N$ whose value will be precised later. Let $a>0$ be such that $\lambda(i)\geq a$ for all $\lambda\in K$ and all $i=1,\dots,d$. 
Then, for all $i=1,\dots,d$, 
\begin{align*}
\left\|\sum_{k=1}^q \frac{1}{w_1(\lambda_k(i))\cdots w_{n_k}(\lambda_k(i))}e_{n_k}\right\|&
\leq \sum_{k=1}^q \frac{1}{\exp(an_k)}\\
&\leq \sum_{j=N_0}^{+\infty}\frac1{\exp(aj)}<\veps
\end{align*}
provided $N_0$ is large enough. Regarding (iii), for all $k=1,\dots,q-1$, all $i=1,\dots,d$ and all $l=0,\dots, N$, 
\begin{align*}
&\left\|\sum_{j=k+1}^{q} \frac{w_{n_j-n_k+l+1}(\lambda_k(i))\cdots w_{n_j+l}(\lambda_k(i))}{w_{l+1}(\lambda_j(i))\cdots
w_{n_j+l}(\lambda_j(i))}e_{n_j-n_k+l}\right\|
\\
&\quad\quad\leq \sum_{j=k+1}^{+\infty}
\exp\big(-\big((\lambda_j(i)n_j-\lambda_k(i)n_k)\big)\big).
\end{align*}
Now,
\begin{align*}
\lambda_j(i)n_j-\lambda_k(i)n_k&=\sum_{s=k}^{j-1}\big(\lambda_{s+1}(i)n_{s+1}-\lambda_s(i)n_s\big)\\
&\geq (j-k)N_0.
\end{align*}
Again, provided $N_0$ is large enough, we get that condition (c) (iii) of Theorem \ref{thm:carac} is satisfied.
\end{proof}

\begin{remark}\label{rem:caracwalpha}
If we work with the family of weight $(w_n(a))_{a>0}$, with $w_1(a)\cdots w_n(a)=\exp(an^\alpha)$, $\alpha\in(0,1)$, we still have a necessary condition for common hypercyclicity if we replace (b) (ii) by: for all $k=1,\dots,q-1$, for all $j=k+1,\dots,q$, for all $i=1,\dots,d$, 
$$(\lambda_j(i)-\lambda_k(i))n_j^\alpha+\lambda_k(i)(n_{j}-n_k)^\alpha>N$$
and this condition is even sufficient on $c_0$.  The lack of linearity when $\alpha\neq 1$ prevents us to go further.
\end{remark}

We are now ready for the proof of Theorem \ref{thm:lipschitz}.
\begin{proof}[Proof of Theorem \ref{thm:lipschitz}]
We may assume that $\Lambda=f([0,1])$ with $f$ satisfying 
$$\exists C>0,\ \forall (s,t)\in [0,1]^2,\ \|f(s)-f(t)\|\leq C|s-t|.$$
Let $\tau>0$, $N\geq 1$ and let $M>0$ be very large (at least $M\geq N$). We set $n_k=kM$ for $k\geq 1$. 
We also define the sequence $(t_k)_{k\geq 1}$ by $t_1=0$ and $t_{k+1}=t_k+\frac{\tau}{C n_k}$. Let $q\geq 1$ be the greatest integer such that $t_q\leq 1$ and define, for $k=1,\dots,q-1$, $I_k=[t_k,t_{k+1}]$
and $I_q=[t_q,1]$. For all $k=1,\dots,q$, we set $\Lambda_k=f(I_k)$ and for $i=1,\dots,d$, we define 
$\lambda_k(i)$ as the maximum of the $i$-th coordinate of the elements of $\Lambda_k$. The Lipschitz condition on $f$ implies that 
$$\forall k=1,\dots,q,\ \Lambda_k\subset\prod_{i=1}^d \left[\lambda_k(i)-\frac{\tau}{n_k},\lambda_k(i)\right]$$
$$\forall k=1,\dots,q-1,\ \forall i=1,\dots,d,\ |\lambda_{k+1}(i)-\lambda_k(i)|\leq C|t_{k+2}-t_k|\leq \frac{2\tau}{kM}.$$
Therefore,
\begin{align*}
\lambda_{k+1}(i)n_{k+1}-\lambda_k(i)n_k
&\geq \lambda_k(i)n_{k+1}-\frac{2\tau}{kM}n_{k+1}-\lambda_k(i)n_k\\
&\geq \lambda_k(i)M -\frac{2\tau}{kM}\times (k+1)M\\
&\geq \lambda_k(i)M-4\tau.
\end{align*}
Provided $M$ is large enough, we can ensure that (c)(ii) of Theorem \ref{thm:caracstandard} is satisfied.
\end{proof}


\section{On the size of the parameter sets for common hypercyclicity}

We will prove a slightly more precise result than Theorem \ref{thm:hausdorffintro}.
 Let us recall the relevant definitions (we refer to 
\cite{Falc} and \cite{Mat95} for more on this subject). 
If $\phi:\mathbb (0,+\infty)\to(0,+\infty)$ is a nondecreasing continuous function
satisfying $\lim_{0^+}\phi=0$ ($\phi$ is called a \emph{dimension function} or a 
\emph{gauge function}),
the \emph{$\phi$-Hausdorff outer measure} of a set $E\subset \mathbb R^d$ is 
$$\mathcal H^{\phi}(E)=\lim_{\veps\to 0}\inf_{r\in R_\veps(E)}\sum_{B\in r}\phi(\textrm{diam}(B)),$$
where $R_\veps(E)$ is the set of (countable) coverings of $E$ with balls $B$ of diameter 
$\textrm{diam}(B)\leq\veps$. 
When $\phi(x)=\phi_s(x)=x^s$, we write for short $\mathcal H^s$ instead of $\mathcal H^{\phi_s}$. 
The \emph{Hausdorff dimension} of a set $E$
is defined by
$$\dim_{\mathcal H}(E):=\sup\{s>0: \mathcal H^s (E)>0\}=\inf\{s>0: \mathcal H^s(E)=0\}.$$

\begin{theorem}\label{thm:hausdorff}
Let $\Lambda\subset\RR^d$, let $(T_\lambda)_{\lambda\in\Lambda}$ be a family of operators acting on the Banach space $X$. Assume that there exist a function $\psi:\NN\to(0,+\infty)$, $v\in X$, $\delta>0$ such that,
for all $\lambda,\mu\in\Lambda$, for all $n\in\NN$ and all $u\in X$ satisfying
$$\|T^n_\lambda u-v\|<\delta\quad\textrm{and}\quad\|T^n_\mu u-v\|<\delta,$$
one has 
$$\|T_\lambda^n u -T_\mu^n u\|\geq \psi(n) \|\lambda-\mu\|.$$
If $\bigcap_{\lambda\in\Lambda}HC(T_\lambda)\neq\varnothing$, then $\mathcal H^\phi(\Lambda)=0$ for any gauge function $\phi$ such that $\sum_n \phi\left(\frac  {2\delta}{\psi(n)}\right)<+\infty$.
\end{theorem}

\begin{proof}
Let $\psi, v, \delta$ be given by the assumptions and let $u\in\bigcap_{\lambda\in\Lambda}HC(T_\lambda)$. 
Define $\Lambda_n=\left\{\lambda\in\Lambda: \|T_\lambda^n u-v\|<\delta \right\}.$
Then $\textrm{diam}(\Lambda_n)\leq 2\delta/\psi(n)$. Indeed, pick $\lambda,\mu\in\Lambda_n$ and observe that
\begin{align*}
\psi(n)\|\lambda-\mu\|&\leq \|T_\lambda^n u-T_\mu^n u\|\\
&\leq \|T_\lambda^n u-v\|+\|T_\mu^n u-v\|\\
&\leq 2\delta.
\end{align*}
In particular, since the condition $\sum_n \phi\left(\frac  {2\delta}{\psi(n)}\right)<+\infty$ implies that $\psi(n)\to+\infty$,
the theorem follows from the fact that, for any $N\geq 1$, $\Lambda\subset\bigcup_{n\geq N}\Lambda_n$.
\end{proof}

Theorem~\ref{thm:hausdorffintro} follows directly from the above theorem by considering $\psi(n)=C n^{\alpha}$ and $\phi(x)=x^s$ for any $s>1/\alpha$. Moreover, we can easily apply this to families of weighted shifts.

\begin{corollary}\label{cor:dimension}
Let $X=\ell_p(\NN)$, $p\in[1,+\infty[$ or $X=c_0(\NN)$.
 Let $I\subset\RR$, let $(w(a))_{a\in I}$ be a family of weights. Assume that there exist $C,\alpha>0$ such that, for any $n\geq 1$, for 
 any $a,b\in I$, 
 $$\left|\sum_{j=1}^n \log w_j(a)-\sum_{j=1}^n\log w_j(b)\right| \geq C n^{\alpha} |a-b|.$$
 Then, for any $\Lambda\subset I^d$, 
 $$ \bigcap_{\lambda\in\Lambda} HC(B_{w(\lambda(1))}\times\cdots\times B_{w(\lambda(d))})\neq\varnothing\implies \dimh(\Lambda)\leq \frac 1\alpha.$$
\end{corollary}
\begin{proof}
 Let $v=(e_0,\cdots,e_0)$. Let $\lambda,\mu\in\Lambda$, $n\in \NN$, $u\in X\times \cdots\times X$ be such that 
 $$\left\|\left(B_{w(\lambda(1))}\times \cdots \times B_{w(\lambda(d))}\right)^nu-v\right\|<\frac 12\textrm{ and }
 \left\|\left(B_{w(\mu(1))}\times \cdots \times B_{w(\mu(d))}\right)^nu-v\right\|<\frac 12.$$
 Let $1\le k\le d$. Looking at the $k$-th coordinate, we get 
 $$ \left|w_1(\lambda(k))\dots w_n(\lambda(k))u_n(k)-1\right|<1/2\textrm{ and }\left|w_1(\mu(k))\dots w_n(\mu(k))u_n(k)-1\right|<1/2.$$
 Now, setting $\veps_n=\left\|\left(B_{w(\lambda(1))}\times \cdots \times B_{w(\lambda(d))}\right)^nu - \left(B_{w(\mu(1))}\times \cdots \times B_{w(\mu(d))}\right)^nu\right\|$, we get
 $$\veps_n \geq 
 |w_1(\lambda(k))\cdots w_n(\lambda(k))-w_1(\mu(k))\cdots w_n(\mu(k))|\cdot |u_n(k)|.$$
 Assume for instance that $w_1(\mu(k))\cdots w_n(\mu(k))\geq w_1(\lambda(k))\cdots w_n(\lambda(k))$. Then 
 \begin{align*}
  \veps_n
  & \geq \left|\frac{w_1(\mu(k))\cdots w_n(\mu(k))}{ w_1(\lambda(k))\cdots w_n(\lambda(k))}-1\right| \cdot w_1(\lambda(k))\cdots w_n(\lambda(k))  \cdot |u_n(k)|\\
  &\geq \frac 12 \left|\frac{w_1(\mu(k))\cdots w_n(\mu(k))}{ w_1(\lambda(k))\cdots w_n(\lambda(k))}-1\right|.
 \end{align*}
Now,
\begin{align*}
  \left|\frac{w_1(\mu(k))\cdots w_n(\mu(k))}{ w_1(\lambda(k))\cdots w_n(\lambda(k))}-1\right|
  &= \exp\left(\sum_{j=1}^n \log w_j(\mu(k))-\sum_{j=1}^n \log w_j(\lambda(k))\right)-1\\
  &\geq \exp\left(Cn^\alpha |\lambda(k)-\mu(k)|\right)-1\\
  &\geq C n^\alpha |\lambda(k)-\mu(k)|.
\end{align*}
We conclude by applying Theorem~\ref{thm:hausdorffintro}
\end{proof}

In the case of the multiples of the backward shift, we get the following result:

\begin{corollary}\label{cor:rolewicz}
 Let $d\geq 1$, $\Lambda\subset(0,+\infty)^d$ and $X=\ell_p(\NN)$, $p\in[1,+\infty)$ or $X=c_0(\NN)$. If $\bigcap_{\lambda\in\Lambda}HC(e^{\lambda(1)}B\times\cdots\times e^{\lambda(d)} B)\neq\varnothing$ then $\dimh(\Lambda)\leq 1$.
\end{corollary}

Fixing $\alpha\in(0,1]$, we can also apply Corollary~\ref{cor:dimension} to the more general case of weights defined by $w_1(a)\cdots w_n(a)=\exp(a n^\alpha)$ or by  $w_n(a)=1+\frac{a}{n^{1-\alpha}}$ for all $n\geq 1$ in order to get Corollary~\ref{cor:sizenalpha}. We point out the following example which will be useful later.

\begin{example}\label{cor:expo}
Let $\alpha\in(0,1]$, and $X=\ell_p(\NN)$, $p\in[1,+\infty)$ or $X=c_0(\NN)$. Let $(w(a))_{a>0}$ be the family of weights defined  by $w_1(a)\cdots w_n(a)=\exp(a n^\alpha)$ (resp. by $w_n(a)=1+\frac{a}{n^{1-\alpha}}$) for all $n\geq 1$. If $\bigcap_{\lambda\in(0,+\infty)^d}HC(B_{w(\lambda(1))}\times\cdots\times B_{w(\lambda(d))})\neq\varnothing$, then $\alpha\leq 1/d$.
\end{example}


\section{A common hypercyclicity criterion in several dimensions}


\subsection{Why this statement? Why this proof?}

This section is purely expository. We intend to explain the assumptions of Theorem \ref{thm:multiprecised} and to talk a few words to introduce its proof. Let $\alpha\in(0,1/2)$ and let $w$ be the weight defined by $w_1(a)\cdots w_n(a)=\exp(an^\alpha)$. In view of Corollary~\ref{cor:expo}, a plausible statement is that $(B_{w(a)}\times B_{w(b)})_{(a,b)\in[1,2]^2}$ admits a common
hypercyclic vector.
To prove it and apply either the Basic Criterion or Theorem~\ref{thm:carac}, we need a covering of $\Lambda=[1,2]^2$.
A natural covering is given by the set $\Gamma_m$ of the closed dyadic cubes of width $2^{-m}$. We have to order these cubes, $\Gamma_m=(\Lambda_k)_{k=1,\dots,4^m}$, to fix $\lambda_k=(a_k,b_k)\in\Lambda_k$ and to associate an increasing sequence $(n_k)_{k=1,\dots,4^m}$ to this covering. Because we are working on $\RR^2$, it is not clear how we have to order the dyadic cubes. Figure \ref{fig:covering} shows three natural candidates.

\begin{center}
 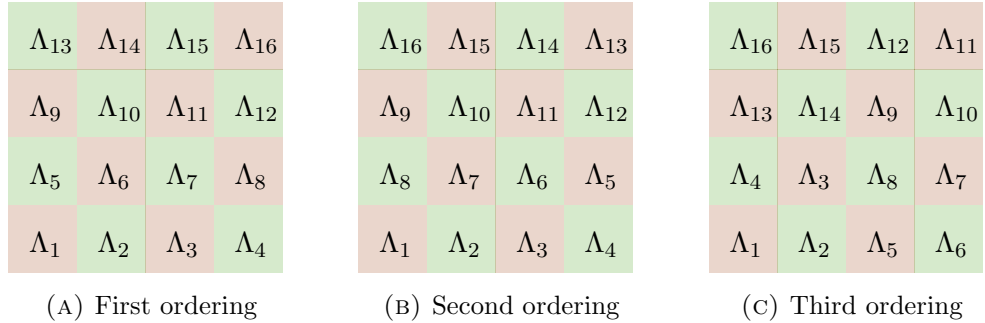
\begin{figure}[h!]\caption{How to order the dyadic covering ($m=2$)} \label{fig:covering}
  \begin{subfigure}[b]{0.25\textwidth}
 \begin{tikzpicture}[line cap=round,line join=round,>=triangle 45,x=0.9cm,y=0.9cm]
\clip(0,-2) rectangle (4,2.5);
\fill[line width=2.pt,color=zzttqq,fill=zzttqq,fill opacity=0.2] (0.,-1.) -- (0.,-2.) -- (1.,-2.) -- (1.,-1.) -- cycle;
\draw (0.15,-1.25) node[anchor=north west,color=black] {$\Lambda_1$};
\fill[line width=2.pt,color=zzttqqa,fill=zzttqqa,fill opacity=0.2] (1.,-1.) -- (1.,-2.) -- (2.,-2.) -- (2.,-1.) -- cycle;
\draw (1.15,-1.25) node[anchor=north west,color=black] {$\Lambda_2$};
\fill[line width=2.pt,color=zzttqq,fill=zzttqq,fill opacity=0.2] (2.,-1.) -- (2.,-2.) -- (3.,-2.) -- (3.,-1.) -- cycle;
\draw (2.15,-1.25) node[anchor=north west,color=black] {$\Lambda_3$};
\fill[line width=2.pt,color=zzttqqa,fill=zzttqqa,fill opacity=0.2] (3.,-1.) -- (3.,-2.) -- (4.,-2.) -- (4.,-1.) -- cycle;
\draw (3.15,-1.25) node[anchor=north west,color=black] {$\Lambda_4$};
\fill[line width=2.pt,color=zzttqq,fill=zzttqqa,fill opacity=0.2] (0.,0.) -- (0.,-1.) -- (1.,-1.) -- (1.,0.) -- cycle;
\draw (0.15,-0.25) node[anchor=north west,color=black] {$\Lambda_5$};
\fill[line width=2.pt,color=zzttqqa,fill=zzttqq,fill opacity=0.2] (1.,0.) -- (1.,-1.) -- (2.,-1.) -- (2.,0.) -- cycle;
\draw (1.15,-0.25) node[anchor=north west,color=black] {$\Lambda_6$};
\fill[line width=2.pt,color=zzttqq,fill=zzttqqa,fill opacity=0.2] (2.,0.) -- (2.,-1.) -- (3.,-1.) -- (3.,0.) -- cycle;
\draw (2.15,-0.25) node[anchor=north west,color=black] {$\Lambda_7$};
\fill[line width=2.pt,color=zzttqqa,fill=zzttqq,fill opacity=0.2] (3.,0.) -- (3.,-1.) -- (4.,-1.) -- (4.,0.) -- cycle;
\draw (3.15,-0.25) node[anchor=north west,color=black] {$\Lambda_8$};
\fill[line width=2.pt,color=zzttqq,fill=zzttqq,fill opacity=0.2] (0.,1.) -- (0.,0.) -- (1.,0.) -- (1.,1.) -- cycle;
\draw (0.15,0.75) node[anchor=north west,color=black] {$\Lambda_9$};
\fill[line width=2.pt,color=zzttqqa,fill=zzttqqa,fill opacity=0.2] (1.,1.) -- (1.,0.) -- (2.,0.) -- (2.,1.) -- cycle;
\draw (1.15,0.75) node[anchor=north west,color=black] {$\Lambda_{10}$};
\fill[line width=2.pt,color=zzttqq,fill=zzttqq,fill opacity=0.2] (2.,1.) -- (2.,0.) -- (3.,0.) -- (3.,1.) -- cycle;
\draw (2.15,0.75) node[anchor=north west,color=black] {$\Lambda_{11}$};
\fill[line width=2.pt,color=zzttqqa,fill=zzttqqa,fill opacity=0.2] (3.,1.) -- (3.,0.) -- (4.,0.) -- (4.,1.) -- cycle;
\draw (3.15,0.75) node[anchor=north west,color=black] {$\Lambda_{12}$};
\fill[line width=2.pt,color=zzttqq,fill=zzttqqa,fill opacity=0.2] (0.,2.) -- (0.,1.) -- (1.,1.) -- (1.,2.) -- cycle;
\draw (0.15,1.75) node[anchor=north west,color=black] {$\Lambda_{13}$};
\fill[line width=2.pt,color=zzttqqa,fill=zzttqq,fill opacity=0.2] (1.,2.) -- (1.,1.) -- (2.,1.) -- (2.,2.) -- cycle;
\draw (1.15,1.75) node[anchor=north west,color=black] {$\Lambda_{14}$};
\fill[line width=2.pt,color=zzttqq,fill=zzttqqa,fill opacity=0.2] (2.,2.) -- (2.,1.) -- (3.,1.) -- (3.,2.) -- cycle;
\draw (2.15,1.75) node[anchor=north west,color=black] {$\Lambda_{15}$};
\fill[line width=2.pt,color=zzttqqa,fill=zzttqq,fill opacity=0.2] (3.,2.) -- (3.,1.) -- (4.,1.) -- (4.,2.) -- cycle;
\draw (3.15,1.75) node[anchor=north west,color=black] {$\Lambda_{16}$};
\end{tikzpicture}
\subcaption{First ordering}
\end{subfigure}
\hspace*{0.6cm}
  \begin{subfigure}[b]{0.25\textwidth}
 \begin{tikzpicture}[line cap=round,line join=round,>=triangle 45,x=0.9cm,y=0.9cm]
\clip(0,-2) rectangle (4,2.5);
\fill[line width=2.pt,color=zzttqq,fill=zzttqq,fill opacity=0.2] (0.,-1.) -- (0.,-2.) -- (1.,-2.) -- (1.,-1.) -- cycle;
\draw (0.15,-1.25) node[anchor=north west,color=black] {$\Lambda_1$};
\fill[line width=2.pt,color=zzttqqa,fill=zzttqqa,fill opacity=0.2] (1.,-1.) -- (1.,-2.) -- (2.,-2.) -- (2.,-1.) -- cycle;
\draw (1.15,-1.25) node[anchor=north west,color=black] {$\Lambda_2$};
\fill[line width=2.pt,color=zzttqq,fill=zzttqq,fill opacity=0.2] (2.,-1.) -- (2.,-2.) -- (3.,-2.) -- (3.,-1.) -- cycle;
\draw (2.15,-1.25) node[anchor=north west,color=black] {$\Lambda_3$};
\fill[line width=2.pt,color=zzttqqa,fill=zzttqqa,fill opacity=0.2] (3.,-1.) -- (3.,-2.) -- (4.,-2.) -- (4.,-1.) -- cycle;
\draw (3.15,-1.25) node[anchor=north west,color=black] {$\Lambda_4$};
\fill[line width=2.pt,color=zzttqq,fill=zzttqqa,fill opacity=0.2] (0.,0.) -- (0.,-1.) -- (1.,-1.) -- (1.,0.) -- cycle;
\draw (0.15,-0.25) node[anchor=north west,color=black] {$\Lambda_8$};
\fill[line width=2.pt,color=zzttqqa,fill=zzttqq,fill opacity=0.2] (1.,0.) -- (1.,-1.) -- (2.,-1.) -- (2.,0.) -- cycle;
\draw (1.15,-0.25) node[anchor=north west,color=black] {$\Lambda_7$};
\fill[line width=2.pt,color=zzttqq,fill=zzttqqa,fill opacity=0.2] (2.,0.) -- (2.,-1.) -- (3.,-1.) -- (3.,0.) -- cycle;
\draw (2.15,-0.25) node[anchor=north west,color=black] {$\Lambda_6$};
\fill[line width=2.pt,color=zzttqqa,fill=zzttqq,fill opacity=0.2] (3.,0.) -- (3.,-1.) -- (4.,-1.) -- (4.,0.) -- cycle;
\draw (3.15,-0.25) node[anchor=north west,color=black] {$\Lambda_5$};
\fill[line width=2.pt,color=zzttqq,fill=zzttqq,fill opacity=0.2] (0.,1.) -- (0.,0.) -- (1.,0.) -- (1.,1.) -- cycle;
\draw (0.15,0.75) node[anchor=north west,color=black] {$\Lambda_9$};
\fill[line width=2.pt,color=zzttqqa,fill=zzttqqa,fill opacity=0.2] (1.,1.) -- (1.,0.) -- (2.,0.) -- (2.,1.) -- cycle;
\draw (1.15,0.75) node[anchor=north west,color=black] {$\Lambda_{10}$};
\fill[line width=2.pt,color=zzttqq,fill=zzttqq,fill opacity=0.2] (2.,1.) -- (2.,0.) -- (3.,0.) -- (3.,1.) -- cycle;
\draw (2.15,0.75) node[anchor=north west,color=black] {$\Lambda_{11}$};
\fill[line width=2.pt,color=zzttqqa,fill=zzttqqa,fill opacity=0.2] (3.,1.) -- (3.,0.) -- (4.,0.) -- (4.,1.) -- cycle;
\draw (3.15,0.75) node[anchor=north west,color=black] {$\Lambda_{12}$};
\fill[line width=2.pt,color=zzttqq,fill=zzttqqa,fill opacity=0.2] (0.,2.) -- (0.,1.) -- (1.,1.) -- (1.,2.) -- cycle;
\draw (0.15,1.75) node[anchor=north west,color=black] {$\Lambda_{16}$};
\fill[line width=2.pt,color=zzttqqa,fill=zzttqq,fill opacity=0.2] (1.,2.) -- (1.,1.) -- (2.,1.) -- (2.,2.) -- cycle;
\draw (1.15,1.75) node[anchor=north west,color=black] {$\Lambda_{15}$};
\fill[line width=2.pt,color=zzttqq,fill=zzttqqa,fill opacity=0.2] (2.,2.) -- (2.,1.) -- (3.,1.) -- (3.,2.) -- cycle;
\draw (2.15,1.75) node[anchor=north west,color=black] {$\Lambda_{14}$};
\fill[line width=2.pt,color=zzttqqa,fill=zzttqq,fill opacity=0.2] (3.,2.) -- (3.,1.) -- (4.,1.) -- (4.,2.) -- cycle;
\draw (3.15,1.75) node[anchor=north west,color=black] {$\Lambda_{13}$};
\end{tikzpicture}
\subcaption{Second ordering}
\end{subfigure}
\hspace*{0.6cm}
  \begin{subfigure}[b]{0.25\textwidth}
 \begin{tikzpicture}[line cap=round,line join=round,>=triangle 45,x=0.9cm,y=0.9cm]
\clip(0,-2) rectangle (4,2.5);
\fill[line width=2.pt,color=zzttqq,fill=zzttqq,fill opacity=0.2] (0.,-1.) -- (0.,-2.) -- (1.,-2.) -- (1.,-1.) -- cycle;
\draw (0.15,-1.25) node[anchor=north west,color=black] {$\Lambda_1$};
\fill[line width=2.pt,color=zzttqqa,fill=zzttqqa,fill opacity=0.2] (1.,-1.) -- (1.,-2.) -- (2.,-2.) -- (2.,-1.) -- cycle;
\draw (1.15,-1.25) node[anchor=north west,color=black] {$\Lambda_2$};
\fill[line width=2.pt,color=zzttqq,fill=zzttqq,fill opacity=0.2] (2.,-1.) -- (2.,-2.) -- (3.,-2.) -- (3.,-1.) -- cycle;
\draw (2.15,-1.25) node[anchor=north west,color=black] {$\Lambda_5$};
\fill[line width=2.pt,color=zzttqqa,fill=zzttqqa,fill opacity=0.2] (3.,-1.) -- (3.,-2.) -- (4.,-2.) -- (4.,-1.) -- cycle;
\draw (3.15,-1.25) node[anchor=north west,color=black] {$\Lambda_6$};
\fill[line width=2.pt,color=zzttqq,fill=zzttqqa,fill opacity=0.2] (0.,0.) -- (0.,-1.) -- (1.,-1.) -- (1.,0.) -- cycle;
\draw (0.15,-0.25) node[anchor=north west,color=black] {$\Lambda_4$};
\fill[line width=2.pt,color=zzttqqa,fill=zzttqq,fill opacity=0.2] (1.,0.) -- (1.,-1.) -- (2.,-1.) -- (2.,0.) -- cycle;
\draw (1.15,-0.25) node[anchor=north west,color=black] {$\Lambda_3$};
\fill[line width=2.pt,color=zzttqq,fill=zzttqqa,fill opacity=0.2] (2.,0.) -- (2.,-1.) -- (3.,-1.) -- (3.,0.) -- cycle;
\draw (2.15,-0.25) node[anchor=north west,color=black] {$\Lambda_8$};
\fill[line width=2.pt,color=zzttqqa,fill=zzttqq,fill opacity=0.2] (3.,0.) -- (3.,-1.) -- (4.,-1.) -- (4.,0.) -- cycle;
\draw (3.15,-0.25) node[anchor=north west,color=black] {$\Lambda_7$};
\fill[line width=2.pt,color=zzttqq,fill=zzttqq,fill opacity=0.2] (0.,1.) -- (0.,0.) -- (1.,0.) -- (1.,1.) -- cycle;
\draw (0.15,0.75) node[anchor=north west,color=black] {$\Lambda_{13}$};
\fill[line width=2.pt,color=zzttqqa,fill=zzttqqa,fill opacity=0.2] (1.,1.) -- (1.,0.) -- (2.,0.) -- (2.,1.) -- cycle;
\draw (1.15,0.75) node[anchor=north west,color=black] {$\Lambda_{14}$};
\fill[line width=2.pt,color=zzttqq,fill=zzttqq,fill opacity=0.2] (2.,1.) -- (2.,0.) -- (3.,0.) -- (3.,1.) -- cycle;
\draw (2.15,0.75) node[anchor=north west,color=black] {$\Lambda_{9}$};
\fill[line width=2.pt,color=zzttqqa,fill=zzttqqa,fill opacity=0.2] (3.,1.) -- (3.,0.) -- (4.,0.) -- (4.,1.) -- cycle;
\draw (3.15,0.75) node[anchor=north west,color=black] {$\Lambda_{10}$};
\fill[line width=2.pt,color=zzttqq,fill=zzttqqa,fill opacity=0.2] (0.,2.) -- (0.,1.) -- (1.,1.) -- (1.,2.) -- cycle;
\draw (0.15,1.75) node[anchor=north west,color=black] {$\Lambda_{16}$};
\fill[line width=2.pt,color=zzttqqa,fill=zzttqq,fill opacity=0.2] (1.,2.) -- (1.,1.) -- (2.,1.) -- (2.,2.) -- cycle;
\draw (1.15,1.75) node[anchor=north west,color=black] {$\Lambda_{15}$};
\fill[line width=2.pt,color=zzttqq,fill=zzttqqa,fill opacity=0.2] (2.,2.) -- (2.,1.) -- (3.,1.) -- (3.,2.) -- cycle;
\draw (2.15,1.75) node[anchor=north west,color=black] {$\Lambda_{12}$};
\fill[line width=2.pt,color=zzttqqa,fill=zzttqq,fill opacity=0.2] (3.,2.) -- (3.,1.) -- (4.,1.) -- (4.,2.) -- cycle;
\draw (3.15,1.75) node[anchor=north west,color=black] {$\Lambda_{11}$};
\end{tikzpicture}
\subcaption{Third ordering}
\end{subfigure}
 \end{figure}
 \end{center}

This order and the associated sequence $(n_k)$ are very important and we know that they at least have to satisfy the following conditions:
\begin{itemize}
\item $n_{4^m}^\alpha$ cannot be greater than $2^m$, so that $\Lambda_k\subset B(\lambda_k,C/n_k^\alpha)$ for some $C>0$ and for all $k=1,\dots,4^m$ (see Theorem \ref{thm:carac} (c) (i));
\item for all $k=1,\dots,4^m-1$ and all $j=k+1,\dots,4^m$, 
$$\left\{
\begin{array}{c}
(a_j-a_k)n_j^\alpha+a_k(n_j-n_k)^\alpha>0\\
(b_j-b_k)n_j^\alpha+b_k(n_j-n_k)^\alpha>0\\
\end{array}\right.$$
(see Remark \ref{rem:caracwalpha}).
\end{itemize}
The last conditions are always satisfied if $a_j>a_k$ and $b_j>b_k$ but are relevant if there is a backward jump between $\Lambda_k$ and $\Lambda_j$, namely if either $a_j<a_k$ or $b_j<b_k$. Suppose for instance that $a_j<a_k$. In that case a small computation shows that we must have
\begin{equation}\label{eq:saut}
n_j\geq \left(\frac1{1-\frac{(a_k-a_j)^{1/\alpha}}{a_k^{1/\alpha}}}\right)n_k\geq \left(\frac1{1-\frac{(a_k-a_j)^{1/\alpha}}{2^{1/\alpha}}}\right)n_k.
\end{equation}
Let us now discuss what this implies on each of the three orderings.
\begin{enumerate}[1.]
\item for the first ordering, there are $2^{m}-1$ backward jumps of size around $1$, say at least $1/2$, between two consecutive dyadic cubes, namely we must have
$$n_{l\cdot 2^m+1}\geq \left(\frac{1}{1-\frac{1}{2^{2/\alpha}}}\right)n_{l\cdot 2^m}$$
for $l=1,\dots,2^m-1$. Hence we will have at least
$$n_{4^m}\geq 	\left(\frac{1}{1-\frac{1}{2^{2/\alpha}}}\right)^{2^{m}-1}n_1$$
which is much bigger than $2^{m/\alpha}$.
\item for the second ordering, we have $2^{m-1}$ backward jumps of size around $1$, say again at least $1/2$, now between the cubes $\Lambda_{(2l+1)\cdot 2^m+1}$ and $\Lambda_{(2l+2)\cdot 2^m}$, for $l=0,\dots,2^{m-1}-1$. Therefore we must have
$$n_{(2l+2)\cdot 2^m}\geq \left(\frac{1}{1-\frac{1}{2^{2/\alpha}}}\right)n_{(2l+1)\cdot 2^m+1},$$
$l=0,\dots,2^{m-1}-1$,
which again implies that $n_{4^m}$ will be much bigger than expected
since
\[n_{4^m}\ge \left(\frac{1}{1-\frac{1}{2^{2/\alpha}}}\right)^{2^{m-1}} n_{2^m+1}.\]
\item An important part of the proof of Theorem \ref{thm:multiprecised} will be to show that the third way to order the covering
is much more economical from this point of view (heuristically speaking, because the big backward jumps are not consecutive, see $\Lambda_{10}$, $\Lambda_{13}$, $\Lambda_{11}$ and $\Lambda_{16}$). More specifically, we will be able to exhibit an increasing sequence $(n_k)$ such that there exists some $D>0$ satisfying
\begin{equation}\label{eq:saut2}
\forall 1\leq k<j\leq 4^m,\ \forall \lambda\in\Lambda_k,\ \forall \mu\in \Lambda_j,\ \|\lambda-\mu\|
\leq D\left(\frac{n_j-n_k}{n_k}\right)^\alpha
\end{equation}
and such that $n_{4^m}^\alpha$ is smaller than $2^m$.
\end{enumerate}
Thus the third way to order the covering is very well adapted to the problem of finding a common hypercyclic vector for the family $(B_{w(a)}\times B_{w(b)})_{(a,b)\in[1,2]^2}$. 
That is why we will use this ordering in the general case, so the assumption (a) in Theorem \ref{thm:multiprecised} becomes very natural.


\subsection{About the homogeneous box dimension}
In this subsection, we discuss the property of having homogeneous box dimension at most $\gamma$. 
We first recall the classical notion of the upper box dimension. Let $\Lambda\subset\mathbb R^d$ be compact. Its 
upper box dimension is defined by
$$\dboxsup(\Lambda)=\limsup_{\veps\to 0}\frac{\log N(\veps)}{\log(1/\veps)},$$
where $N(\veps)$ denotes the smallest number of cubes of size $\veps>0$ which are needed to cover $\Lambda$. We do not change the definition if we only allow $\veps$ to be equal to $c\rho^m$ for some $c>0$, some $\rho\in(0,1)$ and all $m\in\NN$. Namely, for all $\rho\in (0,1)$, 
$$\dboxsup(\Lambda)=\limsup_{m\to+\infty} \frac{\log N(c\rho^m)}{-m\log \rho}.$$

Let us now have a look on the property of having homogeneous box dimension at most $\gamma$. The first two conditions  could be easily rephrased by saying that $\Lambda$ has upper box dimension at most $\gamma$: setting $\rho=1/r^{1/\gamma}$, for each $m\geq 1$, you can cover $\Lambda$ by $1/\rho^{\gamma m}$ balls of radius $C(\Lambda) \rho^m$. The last condition adds the homogeneity requirement: the covering of depth $m$ should be, in a precise sense, a refinement of the covering of depth $m-1$.

It is therefore clear that 
$$\dimh(\Lambda)\leq \dboxsup(\Lambda)\leq \dboxhom(\Lambda).$$
It turns out that, in many cases, one has equality or at least we can prove that $\Lambda$ has homogeneous box dimension at most $\gamma$. Also, any compact subset of $\mathbb R^d$ has homogeneous box dimension at most $d$: we may assume that $\Lambda\subset [0,1]^d$ and we define $r=2^d$ and $\Lambda_{\mathbf k}$ as the intersections of $\Lambda$ with the dyadic subcubes of $[0,1]^d$ with width $2^{-m}$.

We can also provide positive results for compact selfsimilar sets. A compact set $\Lambda\subset\RR^d$ is called {\it selfsimilar} provided
there exists $r$ similarities $s_1,\dots,s_r$ with respective ratio $\rho_1,\dots,\rho_r\in (0,1)$  such that $\Lambda=\bigcup_{i=1}^r s_i(\Lambda)$. For $\bk\in I_r^m$, define $s_\bk=s_{k_1}\circ\cdots\circ s_{k_r}$. 
Let $\gamma$ be defined by 
$$\gamma=\max\left\{\frac{-\log r}{\log(\rho_i)}:\ i=1,\dots,r\right\}.$$
Then setting $\Lambda_\bk=s_\bk(\Lambda)$, one can show that $\Lambda$ has homogeneous box dimension at most $\gamma$. If all the ratios are equal to the same $\rho$, then $\Lambda$ has homogeneous box dimension at most $-\log r/\log \rho$, and when $\Lambda$ satisfies the open set condition (namely there exists $V\subset\RR^d$ open such that $\bigcup_{i=1}^r s_i(V)\subset V$ and $s_i(V)\cap s_j(V)\neq\varnothing$ for $i\neq j$), it is well-known that the Hausdorff dimension of $\Lambda$ equals this value. Hence, in that case
$$\dimh(\Lambda)= \dboxsup(\Lambda)=\dboxhom(\Lambda)=\frac{-\log r}{\log\rho}.$$

Another interesting example is that of H\"older curves. Assume that $\Lambda=f([0,1])$ where $f:[0,1]\to\mathbb R^d$ satisfies
$$\|f(s)-f(t)\|\leq C|s-t|^\alpha,$$
$C>0$, $\alpha\in(0,1)$. We set $r=2$ and for $\bk\in I_{r}^{m}$, we define $I_\bk$ as the dyadic interval
$$I_\bk=\left[\sum_{i=1}^m \frac{k_i-1}{2^i},\sum_{i=1}^m \frac{k_i-1}{2^i}+\frac 1{2^m}\right].$$
Then define $\Lambda_\bk=f(I_{\bk})$. We get immediately that
$$\diam(\Lambda_\bk)\leq C\diam(I_{\bk})^\alpha=C\left(\frac1{2^\alpha}\right)^m.$$
Therefore, $\Lambda$ has homogeneous box dimension at most $1/\alpha$ and there are well-known examples where the box dimension of such a curve (hence, its homogeneous box dimension) is exactly equal to $1/\alpha$.


\subsection{Examples}
Before proceeding with the proof of Theorem \ref{thm:multiprecised}, let us show how this theorem can be applied to a direct sum of weighted shifts. Let $X=\ell_p(\NN)$ or $c_0(\NN)$, $p\in[1,+\infty)$, $I\subset\RR$ compact and $(w(a))_{a\in I}$ be a continuous family of positive weights. We keep the notations of Section \ref{sec:ws}, namely for $\lambda\in I^d$, we denote
by $T_\lambda=B_{w(\lambda(1))}\times \cdots\times B_{w(\lambda(d))}$ and by $S_\lambda=F_{w^{-1}(\lambda(1))}\times
\cdots\times F_{w^{-1}(\lambda(d))}$. We also set $\mathcal D=c_{00}^d$ and we endow $X^d$ with 
$$\|u\|=\max(\|u(1)\|,\dots,\|u(d)\|).$$
 We first point out that (b) of Theorem \ref{thm:multiprecised} is implied by a Lipschitz inequality on $f_n$ where $f_n(a)= \sum_{k=1}^n \log(w_k(a))$ as in Theorem~\ref{thm:carac}. 
Indeed, suppose that  there exist $\alpha>0$ and $C>0$ such that, for all $a,b\in I$, 
\begin{equation}\label{eq:lipschitzalpha}
|f_n(a)-f_n(b)|=\left|\sum_{j=1}^n\big( \log w_j(a)-\log w_j(b)\big)\right|\leq Cn^\alpha |a-b|.
\end{equation}
Then observe first that \eqref{eq:lipschitzalpha} implies that, for all $L>0$, there exists $C'>0$ such that, for all $a,b\in I$, for all $l\in [0,L]$, 
$$\left|\sum_{j=l+1}^{n+l} \big(\log w_j(a)-\log w_j(b)\big)\right|\leq C'n^\alpha |a-b|.$$
Let $u=(u(1),\dots,u(d))\in\mathcal D$ and consider $L>0$ such that the support of each $u(i)$ is contained in $[0,L]$. For all $\lambda,\mu\in I^d$, 
\begin{align*}
\left\|T_\lambda^n S_\mu^n u-u\right \|&\leq \|u\| \max_{i=1,\dots, d}\ \max_{l=0,\dots,L}
\left|\frac{w_{l+1}(\lambda(i))\cdots w_{l+n}(\lambda(i))}{w_{l+1}(\mu(i))\cdots w_{l+n}(\mu(i))}-1\right|\\
&\leq \|u\| \max_{i=1,\dots,d}\ \max_{l=0,\dots,L}\left|\exp\left(\sum_{j=l+1}^{n+l}\big( \log w_j(\lambda(i))-\log w_j(\mu(i))\big)\right)-1\right|\\
&\leq \|u\|  \big(\exp(C'n^\alpha \|\lambda-\mu\|)-1\big).
\end{align*}
Clearly, for all $\veps>0$, we can choose a sufficiently small $\tau>0$ (depending on $\veps$, $\|u\|$, $C'$) for all $n\geq 1$, if $\|\lambda-\mu\|\leq\tau/n^\alpha$, then
$$\|T_\lambda^n S_\mu^n u-u\|<\veps.$$
Hence, it is enough to assume \eqref{eq:lipschitzalpha} to get (b) of Theorem \ref{thm:multiprecised}. Let us now turn to (a), under the assumption \eqref{eq:lipschitzalpha}. 
What we need is the product $w_1(a)\cdots w_n(a)$ not to be too small.
\begin{lemma}
Let $\alpha\in (0,1]$ and assume that there exist $C_1,C_2$ and $C_3>0$ such that 
 \begin{itemize}
    \item $a \in I \mapsto\sum_{j=1}^n\log(w_j(a))$ is $C_1n^\alpha$-Lipschitz;
    \item $\inf_{a\in I}w_1(a)\cdots w_n(a)\geq C_2\exp(C_3n^\alpha).$
\end{itemize}
There exists $D>0$ such that, for all $u\in\mathcal D$, there exist $M>0$ and $N>0$ such that, for all $\lambda,\mu\in I^d$, for all $n\geq 0$ and $k\geq N$ such that $\|\lambda-\mu\|\leq D\frac{k^\alpha}{(n+k)^\alpha}$, then 
\[ \left\| T_\lambda^{n+k}S_\mu^n u\right\|\leq \frac M {k} \quad\textrm{and}\quad \left\| T_\lambda^{n}S_\mu^{n+k} u\right\|\leq \frac M {k}. \]
\end{lemma}
\begin{proof}
Again we fix $L>0$ such that the support of each $u(i)$ is contained in $[0,L]$. 
 Choosing $N>L$, we will have $T_\lambda^{n+k}S_\mu^n u=0$ provided $k\ge N$. On the other hand,
\begin{align*}
\|T_\lambda^n S_\mu^{n+k}u\|&\leq \|u\| \max_{i=1,\dots,d}\ \max_{l=0,\dots, L}
\underbrace{\frac{w_{l+1}(\lambda(i))\cdots w_{l+n+k}(\lambda(i))}{w_{l+1}(\mu(i))\cdots w_{l+n+k}(\mu(i))}}_{F_1}\times\underbrace{\frac 1{w_{l+1}(\lambda(i))\cdots w_{l+k}(\lambda(i))}}_{F_2}\\
\end{align*}
We write
\begin{align*}
F_1&\leq \frac{w_{1}(\lambda(i))\cdots w_{l+n+k}(\lambda(i))}{w_{1}(\mu(i))\cdots w_{l+n+k}(\mu(i))}
\times \frac{w_{1}(\mu(i))\cdots w_{l}(\mu(i))}{w_{1}(\lambda(i))\cdots w_{l}(\lambda(i))}\\
&\leq \exp\big(C_1((l+n+k)^\alpha+l^\alpha)\|\lambda-\mu\|\big)\\
&\leq \exp\big(2C_1 (n+k)^\alpha\|\lambda-\mu\|\big)
\end{align*}
provided $N$, hence $k$, is large enough. If we add the assumption $\|\lambda-\mu\|\leq \frac{Dk^\alpha}{(n+k)^\alpha}$,
we get 
$$F_1\leq \exp(2DC_1 k^\alpha).$$
 On the other hand
\begin{align*}
F_2&\leq \left(\sup_{a\in I}\sup_{l\in [1,L]} \max(1,w_l(a))\right)^L \frac{1}{C_2}\exp(-C_3k^\alpha).
\end{align*}
Hence
$$\|T_\lambda^n S_\mu^{n+k}u\|\leq M\exp((2DC_1-C_3)k^\alpha)$$
for some constant $M$ depending only on $u$ and on the weight, but not on $k$ and $n$.
Thus, we get the result by picking $D$ sufficiently small.
\end{proof}

Summarizing we have obtained the following readable corollary.

\begin{corollary}\label{cor:critere1}
Let $\gamma\in (0,d]$ and let $\Lambda\subset I^d$ be a compact set with homogeneous box
dimension at most $\gamma$. Let $\alpha\in (0,1/\gamma)$ and let $(w(a))_{a\in I}$ be a continuous family of positive weights. Assume that there exist $C_1,C_2,C_3>0$ and $N\geq0$ such that, for all $n\geq N$, \begin{itemize}
    \item $a \in I \mapsto\sum_{j=1}^n\log(w_j(a))$ is $C_1n^\alpha$-Lipschitz;
    \item $\inf_{a\in I}w_1(a)\cdots w_n(a)\geq C_2\exp(C_3n^\alpha).$
\end{itemize}
Then $\bigcap_{\lambda\in \Lambda} HC(B_{w(\lambda(1))}\times\cdots\times B_{w(\lambda(d))})$ is a dense $G_\delta$ subset of $X^d$.
\end{corollary}

Corollary \ref{cor:critere1} yields immediately Corollary \ref{cor:positifnalpha} stated in the introduction.
Combining Example \ref{ex:expo} and Example \ref{cor:expo}, we get the following interesting corollary.

\begin{corollary}
Let $d\ge 1$. There exists a family of operators $(T_a)_{a\in (0,+\infty)}$ on $\ell_p(\NN)$, $p\in[1,+\infty)$, or on $X=c_0(\NN)$ such that $(T_{\lambda(1)}\oplus \cdots \oplus T_{\lambda(d)})_{\lambda\in (0,+\infty)^d}$ admits a common hypercyclic vector but $(T_{\lambda(1)}\oplus \cdots \oplus T_{\lambda(d+1)})_{\lambda\in (0,+\infty)^{d+1}}$ does not.
\end{corollary}
\begin{proof}
Choose $T_a=B_{w(a)}$ with $w_1(a)\cdots w_n(a)=\exp(a n^{\alpha})$
and 
\[\frac{1}{d+1}<\alpha< \frac{1}{d}.\]
\end{proof}

Observe that Example \ref{ex:expo} and Example \ref{cor:expo} do not settle the case $\alpha=1/d$.

\begin{question}
Let $d\geq 2$, $X=\ell_p(\NN)$, $p\in[1,+\infty)$, or $X=c_0(\NN)$ and let $(w(a))_{a>0}$ be the family of weights defined by $w_1(a)\cdots w_n(a) := \exp(a n^{1/d})$. Does $(B_{w(\lambda(1))}\times\cdots\times B_{w(\lambda(d))})_{\lambda(0,+\infty)^d}$ admit a common hypercyclic vector?
\end{question}

Note that it is also possible to have a family of operators $(T_a)_{a\in (0,+\infty)}$ such that for every $d\ge 1$, $(T_{\lambda(1)}\oplus \cdots \oplus T_{\lambda(d)})_{\lambda\in (0,+\infty)^d}$ admits a common hypercyclic vector. 
\begin{example}
Let $d\geq 1$,  $X=\ell_p(\NN)$, $p\in[1,+\infty)$, or on $X=c_0(\NN)$ and $(w(a))_{a>0}$ be the family of weights defined by $w_1(a)\cdots w_n(a) = 2^n n^a$. Then $\bigcap_{\lambda\in(0,+\infty)^d}HC(B_{w(\lambda(1))}\times\cdots\times B_{w(\lambda(d))})\neq\varnothing$.
\end{example}
\begin{proof}
Let $\alpha\in(0,1)$.
By definition, $a\in (0,+\infty) \mapsto\sum_{i=1}^n\log(w_i(a))$ is $\log(n)$-Lipschitz for all $n\geq1$. In particular it is $C_1 n^\alpha$-Lipschitz if $C_1$ is big enough. We then observe that $w_1(a)\cdots w_n(a) = 2^n n^a\geq \exp(C_2 n^\alpha)$ for some sufficiently small $C_2>0$. Hence the result follows from Corollary \ref{cor:critere1}.
\end{proof}

Following the same lines we can generalize the previous example by taking $\rho>1$, $\alpha>0$ and defining $w_1(a)\cdots w_n(a):=\rho^{n^{\alpha}}n^a$. 

\medskip

Corollary~\ref{cor:critere1} can also be applied to products of weighted shifts, in exactly the same way, when $\Lambda\subset\RR^d$ is a $\beta$-H\"older curve, leading to Example \ref{ex:holder} since a $\beta$-H\"older curve
has homogeneous box dimension at most $1/\beta$. When $\beta=1$, this last result is slightly weaker than Theorem \ref{thm:lipschitz}, leading to the following question.

\begin{question}
Let $\Lambda\subset (0,+\infty)^d$ be a $\beta$-H\"older curve for some $\beta\in(0,1)$. Let $(w(a))_{a>0}$ be the family of weights defined by $w_1(a)\cdots w_n(a)=\exp(an^\beta)$ for all $n\geq 1$. Does $(B_{w(\lambda(1))}\times\cdots\times B_{w(\lambda(d))})_{\lambda\in\Lambda}$ admit a common hypercyclic vector?
\end{question}

\medskip

Now let us focus on the case $w_n(a)=1+\frac a n$. The product $w_1(a)\cdots w_n(a)$ behaves like $n^a$, therefore $a\mapsto \sum_{i=1}^n \log(w_i(a))$ is $\log(n)$-Lipschitz. In particular, it is $Cn^\alpha$-Lipschitz for all $\alpha>0$, which means that we may verify (a) of Theorem \ref{thm:multiprecised} with arbitrarily small values of $\beta$. Nevertheless the product $w_1(a)\cdots w_n(a)$ does not grow sufficiently fast in order to apply Corollary \ref{cor:critere1}? This leads us to the forthcoming result, suitable for weights with slow varying weights.

\begin{corollary}\label{cor:critere2}
Let $X=\ell_p(\NN)$, $p\in[1,+\infty)$, or $X=c_0(\NN)$ and let $(w(a))_{a\in I}$ be a continuous family of positive weights. Assume that there exist $C_1,C_2,\kappa>0$ and $N\geq0$ such that, for all $n\geq N$, 
\begin{itemize}
    \item $a\in I  \mapsto\sum_{j=1}^n\log(w_j(a))$ is $C_1\log(n)$-Lipschitz;
    \item $\inf_{a\in I} w_1(a)\cdots w_n(a)\geq C_2n^\kappa.$
\end{itemize}
Then  $\bigcap_{\lambda\in I^d}HC(B_{w(\lambda(1))}\times\cdots\times B_{w(\lambda(d))})\neq\varnothing$.
\end{corollary}

\begin{proof}
We follow the proof of Corollary \ref{cor:critere1}. Let $C_1,C_2,\kappa,N$ be given by the assumptions. Let $\alpha\in(0,1/d)$ be such that $\kappa/\alpha>d$ and let $\beta\in(0,\kappa)$ be such that $\beta/\alpha>d$. Condition (b) of Theorem \ref{thm:multiprecised} is clearly satisfied.  To prove (a), we keep the same notations. Provided $\|\lambda-\mu\|\leq \frac{Dk^\alpha}{(n+k)^\alpha}$, we may write
\begin{align*}
F_1&\leq \exp\big(2C_1\log(n+k)\|\lambda-\mu\|\big)\\
&\leq \exp\left (2DC_1\log(n+k) \frac{k^\alpha}{(n+k)^\alpha}\right).
\end{align*}
Now, provided $k$ is large enough
(we require $\alpha\log k>1$), the function $n\mapsto \frac{\log(n+k)}{(n+k)^\alpha}$ is decreasing on 
$[0,+\infty)$, so that
$$F_1\leq \exp(2DC_1\log k)=k^{2DC_1}.$$
This implies that
$$F_1F_2\leq C' k^{2DC_1-\kappa}\leq C' k^{-\beta}$$
provided $D$ has been chosen so small that $\beta+2DC_1<\kappa$. Hence, condition (a) of Theorem \ref{thm:multiprecised} is also satisfied.
\end{proof}

\begin{example}
Let $X=\ell_p(\NN)$, $p\in[1,+\infty)$, or $X=c_0(\NN)$ and let $(w(a))_{a>0}$ be the family of weights defined by $w_n(a) = 1 + \frac{a}{n}$ (resp.  $w_n(a) := \big(1 + \frac{1}{n}\big)^a$). Then $\bigcap_{\lambda\in(0,+\infty)^d}HC(B_{w(\lambda(1))}\times\cdots\times B_{w(\lambda(d))})\neq\varnothing$.
\end{example}


Let us also show how we may apply Theorem \ref{thm:multiprecised} to get a common hypercyclic vector for $(e^{a}B\times e^{b}B)_{(a,b)\in\Lambda}$ with $\Lambda$ a classical fractal set. 

\begin{example}\label{cor:cantor}
Let $X=\ell_p(\NN)$, $p\in[1,+\infty)$, or $X=c_0(\NN)$. Let $\Lambda$ be a homogeneous Cantor subset of $(0,+\infty)^2$ with dissection ratio $\rho\in(0,1/4)$. Then $\bigcap_{(a,b)\in\Lambda}HC(e^a B\times e^b B)\neq\varnothing$. 
\end{example}

\begin{proof}
We may apply Corollary \ref{cor:critere1} since $\Lambda$ has homogeneous box dimension at most $-\ln 4/\ln \rho<1$ (we apply the definition with $r=4$).  
\end{proof}

Corollary \ref{cor:rolewicz} and Example \ref{cor:cantor} leave open the case of the Cantor set with dissection ratio $\rho=1/4$. More generally, for $\Lambda$ a compact subset of $(0,+\infty)^2$, we know that 
$$\dimh(\Lambda)> 1\implies \bigcap_{(a,b)\in\Lambda}HC(e^{a} B\times e^{b}B)=\varnothing$$
$$\text{and}\quad\dboxhom(\Lambda)<1\implies  \bigcap_{(a,b)\in\Lambda}HC(e^{a} B\times e^{b}B)\neq\varnothing.$$

It is natural to ask whether we can go further. In the first implication, we cannot replace the Hausdorff dimension by the homogeneous box dimension.

\begin{proposition}
There exists a compact subset $\Lambda\subset (0,+\infty)^2$ with $\dim_{HB}(\Lambda)=2$ such that
$\bigcap_{(a,b)\in\Lambda}HC(e^{a}B\times e^{b}B)\ne \varnothing$.
\end{proposition}
\begin{proof}
Let $I=[1,2]\times\{1\}$ and for any $n\ge 1$, any $0< k< 2^n$, $I_{n,k}=\{1+\frac{k}{2^n}\}\times [1,1+\frac{1}{n}]$. We consider
\[\Lambda=I\cup \bigcup_{n\ge 1}\bigcup_{0< k< 2^n} I_{n,k}.\]
We first remark that $\Lambda$ is closed and thus compact. Let $(\lambda_m,\mu_m)_{m\ge 1}\subset \Lambda$ be a sequence converging to $(\lambda,\mu)\in [1,2]^2$.
If $\mu_m=1$ infinitely often then $\mu=1$ and thus $(\lambda,\mu)\in \Lambda$. If we now assume that $\mu_m\ne 1$ for any $m$, then $(\lambda_m,\mu_m)\in I_{n_m,2k_m+1}$ for a unique $n_m\ge 1$ and a unique $0\le k_m\le 2^{n_m-1}-1$. In particular, $\mu_m\in [1,1+\frac{1}{n_m}]$. Therefore, if $\sup_{m}n_m=\infty$, we get $(\lambda,\mu)\in \Lambda$ since $\mu=1$ and if $\sup_{m}n_m<\infty$, up to an extraction, the sequences $(n_m)_{m\ge 1}$ and $(k_m)_{m\ge 1}$ are ultimately constant and $(\lambda,\mu)\in \Lambda$ since each $I_{n,k}$ is closed.

Since $\Lambda$ is a countable union of Lipschtiz curves, it is a consequence of Theorem \ref{thm:lipschitz} that $\bigcap_{(a,b)\in\Lambda}HC(e^{a}B\times e^{b}B)\ne \varnothing$. It remains to show that $\dim_{HB}(\Lambda)=2$. Note that it suffices to show that $\overline{\dim}_{B}(\Lambda)\ge 2$.
Let $m\ge 2$. How many cubes of size $\frac{1}{2^m}$ are needed to cover $\Lambda$? 
To cover each fiber $I_{m-1,k}$, $0<k<2^{m-1}$, we need at least $2^m/(m-1)$ cubes of size $2^{-m}$. Note that such a cube cannot intersect another fiber of the same generation $I_{m-1,l}$ with $l\neq k$. Therefore, in order to cover $\bigcup_{0<k<2^{m-1}}I_{m-1,k}$, and thus $\Lambda$, we need at least $(2^{m-1}-1)\cdot 2^m/(m-1)$ cubes of size $2^{-m}$. We conclude that
\begin{align*}
N(2^{-m})\geq c \frac{4^m}m
\end{align*}
for some $c>0$. Therefore, 
\[\overline{\dim}_{B}(\Lambda)\ge \lim_{m\to+\infty}\frac{m\log 4-\log m}{m\log 2}=2.\]

\end{proof}

However the following question is open.

\begin{question}
Let $\Lambda$ be a compact subset of $(0,+\infty)^2$ such that $\dimh(\Lambda)<1$. Does the family $(e^{a} B\times e^{b}B)_{(a,b)\in\Lambda}$ admit a common hypercyclic vector?
\end{question}

At least, we can show that the condition $\dimh(\Lambda)\leq 1$ is not sufficient to obtain a common hypercyclic vector.

\begin{proposition}
 There exists a set $\Lambda\subset(0,+\infty)^2$ such that $\dimh(\Lambda)=1$ and nevertheless $\bigcap_{(a,b)\in\Lambda}HC(e^{a} B\times e^{b}B)=\varnothing$.
\end{proposition}
\begin{proof}
 Let $\phi(x)=x/\log^2(x)$. Applying Theorem \ref{thm:hausdorff} as in the proof of Corollary \ref{cor:dimension}, we know that
 $\mathcal H^\phi(\Lambda)=0$ for any $\Lambda\subset(0,+\infty)^2$ such that $\bigcap_{(a,b)\in\Lambda}HC(e^{a} B\times e^{b}B)\neq\varnothing$.
 Consider now for $\Lambda$ the Cantor set starting from $[1,2]^2$ and with non-constant dissection ratio $\frac14\times\frac{(j+1)^2}{j^2}$. Namely, $\Lambda=\bigcap_{m\geq 1}\Lambda_m$
 where $\Lambda_m$ consists in $4^m$ squares of width $ \left(\frac 14\right)^m (m+1)^2$. Then $\dimh(\Lambda)=1$ and using the mass transference principle as in 
 \cite[Example 4.3]{Falc}, $\mathcal H^\phi(\Lambda)>0$. Hence, $\bigcap_{(a,b)\in\Lambda}HC(e^{a} B\times e^{b}B)=\varnothing$.
\end{proof}

For this last example, it is easy to show that one also has $\dboxhom(\Lambda)=1$.


\subsection{A lemma on sequences of integers}

We now proceed with the proof of Theorem \ref{thm:multiprecised}. Let us start with $\Lambda$ a compact subset of $\mathbb R^d$ with homogeneous box dimension at most $\gamma\in(0,d]$. In order to apply the Basic Criterion, we will need a covering of $\Lambda$. Natural coverings are given by the definition of the homogeneous box dimension, namely by the sets $(\Lambda_\bk)_{\bk\in I_r^m}$ for a given value of $m$. As pointed out above,
the way we order these sets is very important. We will choose the ordering obtained by endowing $I_r^m$ with its natural lexicographic order:
\begin{align*}
(i_1,\dots,i_m)<(j_1,\dots,j_m)&\iff \exists p\in\{1,\dots,m\},\ i_1=j_1,\dots,i_{p-1}=j_{p-1}\textrm{ and }j_p>i_p.
\end{align*}

We first define the sequence $(n_{\bk})_{\bk\in I_r^m}$.

\begin{lemma}\label{lem:sequence}
Let $\alpha>0$, $\rho\in(0,1)$ and $r\geq 2$ be such that $\rho^{1/\alpha}r<1$. Then there exist
$c_1>0$ and $c_2>0$ such that, for all $m\geq 1$, for all $n_1\geq 1$, for all $A>0$, the sequence $(n_\bk)_{\bk\in\Imr}$ defined by 
\begin{align*}
n_{1,\dots,1}&=n_1\\
n_{k_1,\dots,k_m}&=\left\lfloor \frac{1}{1-\rho^{p/\alpha}}n_{k_1,\dots,k_{p}-1,r,\dots,r}\right\rfloor +A
\end{align*}
for $p=1,\dots,m$, $k_{p+1}=\dots=k_m=1$,  $k_p\neq 1$ satisfies
$$n_{r,\dots,r}\leq c_1 n_1+c_2 r^m A.$$
\end{lemma}

A key point of this lemma is that $c_1$ and $c_2$ depend neither on $m$ nor on $n_1$ nor on $A$. 
We will do the proof by induction on $m$. Nevertheless, we need to introduce auxiliary sequences to keep track of the involved constants at each step.
\begin{lemma}\label{lemmabis}
Let $\alpha>0$, $\rho\in(0,1)$ and $r\geq 2$.  Let also $B\in (0,1]$, $m\geq 1$, $n_1\geq 1$ and $A>0$. Then the sequence $(n_\bk)_{\bk\in\Imr}$ defined by
\begin{align*}
n_{1,\dots,1}&=n_1\\
n_{k_1,\dots,k_m}&=\left\lfloor \frac{1}{1-B^{1/\alpha}\rho^{p/\alpha}}n_{k_1,\dots,k_{p}-1,r,\dots,r}\right\rfloor +A
\end{align*}
for $p=1,\dots,m$, $k_{p+1}=\dots=k_m=1$,  $k_p\neq 1$ satisfies
$$n_{r,\dots,r}\leq C(m,B) n_1+ D(m,B) A,$$
where
 \begin{align*}
&C(1,B)=\left(\frac1{1-B^{1/\alpha}\rho^{1/\alpha}}\right)^{r-1}, \\
&D(1,B)=r\left(\frac1{1-B^{1/\alpha}\rho^{1/\alpha}}\right)^{r-1}, \\
&C(m,B)= \left(\frac1{1-B^{1/\alpha}\rho^{1/\alpha}}\right)^{r-1} C(m-1, B\rho)^r, \\
&D(m,B)= r \left(\frac1{1-B^{1/\alpha}\rho^{1/\alpha}}\right)^{r-1} C(m-1,B\rho)^{r-1}\big(1+D(m-1,B\rho)\big).
\end{align*}
\end{lemma}

\begin{proof}
We proceed by induction on $m$. To simplify the notation, let $q_B:=\frac{1}{1-B^{1/\alpha}\rho^{1/\alpha}}$. The first step $m=1$ is easy. 
Indeed, for $k=1,\dots,r-1$, we have
$$n_{k+1}\leq  q_B n_k+A$$
what leads to 
$$n_{r}\leq q_B^{r-1}n_1+\sum_{j=0}^{r-2}q_B^j A,$$
which itself gives the (nonoptimal) values for $C(1,B)$ and $D(1,B)$ as in the statement. 

Let us now assume that the property is true at rank $m-1$ and let us verify it at rank $m$. For
 $\bk\in I_r^{m-1}$ and $i\in\{1,\dots,r\}$, define
$$m_\bk(i):=n_{i,k_1,\dots,k_{m-1}}$$
and observe that, for $p=1,\dots,m-1$ and $k_p\neq 1$,
\begin{align*}
m_{k_1,\dots,k_p,1,\dots,1}(i)&=n_{i,k_1,\dots,k_p,1,\dots,1}\\
&=\left\lfloor \frac 1{1-B^{1/\alpha}\rho^{(p+1)/\alpha}}n_{i,k_1,\dots,k_{p}-1,r,\dots,r}\right\rfloor+A\\
&=\left\lfloor \frac 1{1-(B\rho)^{1/\alpha}\rho^{p/\alpha}}m_{k_1,\dots,k_{p}-1,r,\dots,r}(i)\right\rfloor+A.\\
\end{align*}
Therefore, the induction hypothesis yields, for each $i=2,\dots,r$, 
\begin{align*}
n_{i,r,\dots,r}&\leq C(m-1,B\rho) n_{i,1,\dots,1}+D(m-1,B\rho)A\\
&\leq q_B C(m-1,B\rho)n_{i-1,r,\dots,r}+\big(C(m-1,B\rho)+D(m-1,B\rho)\big)A.
\end{align*}
Hence, proceeding as in the initial step and using a last time the induction hypothesis for $i=1$, we find
\begin{align*}
n_{r,\dots,r}&\leq \big(q_BC(m-1,B\rho)\big)^{r-1}n_{1,r,\dots,r}\\
&\quad\quad\quad+(r-1)q_B^{r-1}\big(C(m-1,B\rho)+D(m-1,B\rho)\big)A\\
&\leq q_B^{r-1}C(m-1,B\rho)^r n_1+r q_B^{r-1}C(m-1,B\rho)^{r-1}\big(1+D(m-1,B\rho)\big)A.
\end{align*}
\end{proof}

\begin{proof}[Proof that Lemma \ref{lemmabis} $\Rightarrow$ Lemma \ref{lem:sequence}]
A simple induction yields
\[C(m,1)\leq \prod_{j=1}^m\bigg(\frac{1}{1-\rho^{j/\alpha}}\bigg)^{(r-1)\cdot r^{j-1}}\leq \prod_{j=1}^\infty\bigg(\frac{1}{1-\rho^{j/\alpha}}\bigg)^{(r-1)\cdot r^{j-1}}=:c_1,\]
the last infinite product being convergent by the assumption $\rho^{1/\alpha}r<1$.
 More precisely, we have
\begin{align*}
\log C(m,B)& \leq -(r-1) \sum_{j=1}^m \log(1-B^{1/\alpha}\rho^{j/\alpha})r^{j-1}\\
&\leq (r-1)\cdot C\sum_{j=1}^m B^{1/\alpha} \rho^{j/\alpha} r^{j-1}\\
&\leq C'B^{1/\alpha}
\end{align*}
for some constants $C,C'>0$ which only depend on $\rho$, $\alpha$ and $r$ (recall that $B\in(0,1]$). We use this bound to estimate $D(m,B)$:
\[D(m,B)\leq  r C(m,B) (1+D(m-1,B\rho))\le r\exp(C'B^{1/\alpha})D(m-1,B\rho) +r\exp(C'B^{1/\alpha}).\]
By another induction, we get
\begin{align*}
D(m,B)
&\leq r^{m-1}\exp\bigg(C'B^{1/\alpha}\sum_{j=0}^{m-2}\rho^{j/\alpha}\bigg)D(1,B\rho^{m-1})\\
 &\quad\quad\quad\quad+r\sum_{j=0}^{m-2}r^j \exp\bigg(C'B^{1/\alpha}\sum_{i=0}^{j}\rho^{i/\alpha}\bigg).
\end{align*}
The convergence of $\sum_j \rho^{j/\alpha}$ yields the existence of $c_2$, depending only on $\alpha$, $\rho$ and $r$, such that $D(m,1)\leq c_2 r^m$.
\end{proof}


\subsection{A covering lemma}
We now produce the desired covering together with the sequence of integers. We thus fix $\Lambda$ a compact subset of $\RR^d$  with homogeneous box dimension at most $\gamma\in (0,d]$. Let $r\geq 2$ and $C(\Lambda)>0$ be such that, for all $m\geq 1$, one can construct the compact sets $(\Lambda_\bk)_{\bk\in I_r^m}$ as in Definition \ref{def:homogeneousdimension}. We also fix $\rho=\left(\frac 1r\right)^{1/\gamma}$, $D>0$, $\alpha\in (0,1/\gamma)$ and $\beta>\alpha\gamma$. Let $c_1$, $c_2$ be the constants given by Lemma~\ref{lem:sequence}. We will assume in this subsection that 
\begin{equation}\label{eq:diamlambda}
C(\Lambda)\leq \frac{D}{(2c_1)^\alpha r^{1/\gamma}}.
\end{equation}

\begin{lemma}\label{lem:covering}
For all $\tau>0$, for all $\delta>0$, for all $N\geq 1$, there exist $q\geq 1$, an increasing sequence of integers $(n_k)_{k=1,\dots,q}$,
a sequence of parameters $(\lambda_k)_{k=1,\dots,q}\subset\Lambda$, a sequence $(\Lambda_k)_{k=1,\dots,q}$ of compact subsets
of $\Lambda$ such that
\begin{enumerate}[(a)]
\item $n_1\geq N$, $n_{j+1}-n_j\geq N$;
\item  $\Lambda=\bigcup_{k=1}^{q}\Lambda_k$ and, for all $k=1,\dots,q$, $\Lambda_k\subset \bar B(\lambda_k,\tau/n_k^\alpha)$;
\item for all $1\leq k<j\leq q$, for all $\lambda\in \Lambda_k$, for all $\mu\in \Lambda_j$, 
$$\|\lambda-\mu\| \leq \frac{D(n_j-n_k)^\alpha}{n_j^\alpha};$$
\item for all $k\in\{1,\dots,q\}$, $\sum_{j\neq k}\frac{1}{|n_j-n_k|^\beta}\leq\delta$;
\item $\sum_{j=1}^{q}\frac 1{n_j^\beta}\leq \delta$.
\end{enumerate}
\end{lemma}
\begin{proof}
The inequality $\alpha\gamma<\beta$ implies that $\rho^{\beta/\alpha}r<1$.
We consider $\kappa>0$, $s>0$, $A\geq N$ and $m\geq 1$ satisfying the following constraints:
$$\kappa:=\frac{\tau^{1/\alpha}}{4c_1(C(\Lambda))^{1/\alpha}},$$
$$r\kappa^{-\beta} \sum_{p=s}^{+\infty}(r\rho^{\beta/\alpha})^p< \frac{\delta}3,$$
$$\sum_{l=1}^{r^{s+1}}\frac{1}{l^\beta A^\beta}<\frac{\delta}3,$$
$$\left\lfloor \frac{1}{2c_1}\left(\frac{\tau}{\rho^m C(\Lambda)}\right)^{1/\alpha}\right\rfloor 
\geq \max\left(\left(\frac3\delta\right)^{1/\beta},2+\frac{c_2}{c_1}Ar^m,N\right).$$

Observe that the conditions imposed on $\rho$, $\alpha$, $\beta$ and $r$ allow us to define successively $s$, $A$ and $m$. We set 
$$n_1:=\left\lfloor \frac{1}{2c_1}\left(\frac{\tau}{\rho^m C(\Lambda)}\right)^{1/\alpha}\right\rfloor$$
and we consider the sequence $(n_\bk)_{\bk\in\Imr}$ defined by Lemma \ref{lem:sequence}. We first remark that (a) is satisfied since $n_1\ge N$ and $n_{j+1}-n_j\ge A$ for every $j\ge 1$. We then set $q=r^m$ and we observe that
\begin{align*}
\max(n_\bk:\bk\in\Imr)&=n_{r,\dots,r}\\
&\leq \frac 12\left(\frac{\tau}{\rho^m C(\Lambda)}\right)^{1/\alpha}+c_2 A r^m\\
&\leq \left(\frac{\tau}{\rho^m C(\Lambda)}\right)^{1/\alpha}-2c_1\\
&\leq 2c_1n_1\leq \left(\frac{\tau}{\rho^m C(\Lambda)}\right)^{1/\alpha}.
\end{align*}

We then consider the covering $(\Lambda_\bk)_{\bk\in I_r^m}$ of $\Lambda$ given by Definition \ref{def:homogeneousdimension} and we fix $\lambda_{\bk}\in\Lambda_{\bk}$. 
Since
$$\diam(\Lambda_\bk)\leq \rho^m C(\Lambda)\leq \frac{\tau}{n_\bk^\alpha},$$
we get (b). Let now $\bk,\bj\in\Imr$ with $\bk<\bj$ and let $\lambda\in\Lambda_\bk$, $\mu\in\Lambda_\bj$. 
Let $p$ be the biggest integer such that $k_1=j_1,\dots,k_{p-1}=j_{p-1}$ so that $k_p<j_p$. Then $\lambda$
and $\mu$ both belong to $\Lambda_{k_1,\dots,k_{p-1}}$ which has diameter less than $\rho^{p-1}C(\Lambda)$. On the other hand the definition of the sequence $(n_{\mathbf i})$ ensures that
\begin{align*}
n_\bj-n_\bk&\geq \left(\frac1{1-\rho^{p/\alpha}}-1\right) n_\bk\\
&\geq \rho^{p/\alpha}n_\bk\\
&\geq \frac{\rho^{p/\alpha}}{2c_1}n_\bj
\end{align*}
so that  (c) is satisfied, since
\begin{align*}
\|\lambda-\mu\|&\leq \frac{C(\Lambda)}{\rho}\rho^p\\
&\leq (2c_1)^\alpha \frac{C(\Lambda)}{\rho}\left(\frac{n_\bj-n_\bk}{n_\bj}\right)^{\alpha}\\
&\leq D\left(\frac{n_\bj-n_\bk}{n_\bj}\right)^{\alpha}.
\end{align*}
 Let us now prove (d) and (e). For $\bj\neq\bk\in\Imr$, we denote by $\gamma(\bj,\bk)$ the biggest integer $p$ such that $k_1=j_1,\dots,k_{p-1}=j_{p-1}$, with $\gamma(\bj,\bk)=1$ if $k_1\neq j_1$. In particular, if we fix $\bk\in\Imr$ and $p\in\{1,\dots,m\}$, we can observe that
$$\card\left(\left\{\bj\in\Imr:\ \gamma(\bj,\bk)=p\right\}\right)\leq r^{m+1-p}.$$
Moreover, if $\gamma(\bj,\bk)=p$, then the computation done above shows that 
\begin{align*}
|n_\bj-n_\bk|&\geq \rho^{p/\alpha}n_1\\
&\geq \kappa \rho^{-(m-p)/\alpha}.
\end{align*}
We then split the sum appearing in (d) into two parts. On the one hand, using this last estimation,
\begin{align*}
\sum_{\substack{\bj\neq\bk\\ \gamma(\bj,\bk)\leq m-s}}\frac1{|n_\bj-n_\bk|^\beta}&\leq \sum_{p=1}^{m-s}\sum_{\substack{\bj\neq\bk\\ \gamma(\bj,\bk)=p}}\kappa^{-\beta}\rho^{(m-p)\beta/\alpha}\\
&\leq \sum_{p=1}^{m-s}r \kappa^{-\beta} (r\rho^{\beta/\alpha})^{m-p}\\
&\leq r\kappa^{-\beta}\sum_{p=s}^{+\infty} (r\rho^{\beta/\alpha})^p<\frac\delta3
\end{align*}
by the choice of $s$. On the other hand, we observe that there are at most 
$r+\dots+r^s\leq r^{s+1}$ elements $\bj\in\Imr$ such that $\bj\neq\bk$ and $\gamma(\bj,\bk)\geq m-s+1$. 
Moreover, the difference between two consecutive terms of the sequence $(n_\bj)$ is at least $A$. Thus,
$$\sum_{\substack{\bj\neq \bk\\ \gamma(\bj,\bk)\geq m-s+1}}\frac1{|n_\bj-n_\bk|^\beta}\leq
\sum_{l=1}^{r^{s+1}} \frac1{l^\beta A^\beta}<\frac\delta 3.$$
This achieves the proof of (d) with the stronger bound $2\delta/3$. Moreover, we can use this improved estimate to get easily (e):
$$\sum_{\bk\in\Imr}\frac1{n_\bk^\beta}\leq \frac{1}{n_{1,\dots,1}^\beta}+\sum_{\substack{\bk\in\Imr \\ \bk>(1,\dots,1)}}\frac 1{|n_\bk-n_{1,\dots,1}|^{\beta}}<\delta.$$
\end{proof}


\subsection{Proof of Theorem \ref{thm:multiprecised} }

\begin{proof}
We shall prove that the assumptions of the Basic Criterion are satisfied. Let $r\geq 2$ be such that, for all $m\geq 0$, there exists a sequence of compact sets $(\Lambda_\bk)_{\bk\in I_r^m}$ satisfying the assumptions of Definition \ref{def:homogeneousdimension}.
Since for each $m\geq 1$ and each $\bk\in I_r^m$, the set $\Lambda'=\Lambda_\bk$ satisfies the same assumptions as $\Lambda$ with $C(\Lambda')=C(\Lambda)\left(\frac 1{r^{1/\gamma}}\right)^m$ (just define, for $\bj\in I_r^{m'}$, $\Lambda'_\bj=\Lambda_{\bk,\bj}$) and since the assumptions of Theorem~\ref{thm:multiprecised} are satisfied by $\Lambda'$ for the same constants $\alpha$, $\beta$ and $D$, we may assume that 
$$C(\Lambda)\leq \frac{D}{(2c_1)^\alpha r^{1/\gamma}}.$$
Let $\veps>0$, $u,v\in\mathcal D$. Let $C,\tau>0$ and $N\in\mathbb N$ be such that the assumptions of Theorem \ref{thm:multiprecised} are satisfied for both $u$ and $v$. 
We then consider the sequences $(n_k)$, $(\lambda_k)$ and $(\Lambda_k)$ given by Lemma \ref{lem:covering} applied with $\tau,N$ and $\delta=\veps/C$ (we may always assume that $C\geq 1$). It is now an easy exercise to prove that the assumptions of the Basic Criterion are satisfied. The most difficult point is to prove that, for all $k\in\{1,\dots,q\}$, for all $\lambda\in \Lambda_k$, one has
$$\left\|\sum_{j\neq k} T_\lambda^{n_k}S_{\lambda_j}^{n_j}(v)\right\|<\veps.$$
When $j>k$, 
$$\|\lambda-\lambda_j\|\leq D\frac{(n_j-n_k)^\alpha}{n_j^\alpha}$$
so that
$$\left\|\sum_{j>k} T_\lambda^{n_j}S_{\lambda_k}^{n_k} v\right\|\leq \sum_{j>k}\frac{C}{(n_j-n_k)^\beta},$$
whereas, when $j<k$, 
$$\|\lambda-\lambda_j\| \leq D\frac{(n_k-n_j)^\alpha}{n_k^\alpha}$$
so that 
$$\left\|\sum_{j<k} T_\lambda^{n_k}S_{\lambda_j}^{n_j} v\right\|\leq \sum_{j<k}\frac{C}{(n_k-n_j)^\beta}.$$
Property (d) of Lemma \ref{lem:covering} now finishes the job.
\end{proof}

\begin{question}
Condition (b) of Theorem \ref{thm:multiprecised} does not perfectly match (CS2) of the Costakis-Sambarino theorem because we cannot take $\alpha=1/\gamma$. Is it possible to cover this last case (changing if necessary condition (a))?
\end{question}
Observe that in the previous proof the condition $\alpha<1/\gamma$ was needed to get the convergence of the infinite product defining $C(m,1)$.

\end{document}